\documentclass[lectures,reqno]{pcms-l}

\usepackage[capitalize]{cleveref}

\usepackage{amssymb}
\usepackage{fancybox}
\usepackage{graphicx}
\usepackage[mathscr]{eucal}
\usepackage{url}

\theoremstyle{plain}
\newtheorem{thm}[equation]{Theorem}
\newtheorem{prop}[equation]{Proposition}
\newtheorem{cor}[equation]{Corollary}

\newtheorem{lem}[equation]{Lemma}

\newtheorem{conj}[equation]{Conjecture}
\newtheorem{open}[equation]{Open Problem}

\newtheorem{unfortunate}[equation]{Fact}

\crefformat{prop}{Proposition~#2#1#3}
\crefname{prop}{proposition}{propositions}
\Crefname{prop}{Proposition}{Propositions}
\crefname{section}{section}{sections}
\Crefname{section}{Section}{Sections}
\crefformat{section}{Section~#2#1#3}

\Crefformat{exer}{Exercise~#2#1#3}
\Crefname{exer}{Exercise}{Exercises}

\Crefname{conj}{Conjecture}{Conjectures}

 \newcounter{step}
 \newenvironment{step}[1][\unskip]{\refstepcounter{step}\em
 \medskip \noindent Step \thestep\ #1.\ }{\unskip\upshape}
 \renewcommand{\thestep}{\arabic{step}}

\crefformat{step}{Step~#2#1#3}
\crefmultiformat{step}{Steps~#2#1#3}{ and~#2#1#3}{, #2#1#3}{}
\crefmultiformat{thm}{Theorems~#2#1#3}{ and~#2#1#3}{, #2#1#3}{}

 \newcounter{case}
 \newenvironment{case}[1][\unskip]{\refstepcounter{case}\it
 \medskip \noindent Case \thecase\ #1. }{\unskip\upshape}
 \renewcommand{\thecase}{\arabic{case}}

\theoremstyle{remark}
\newtheorem{rem}[equation]{Remark}
\newtheorem{rems}[equation]{Remarks}
\newtheorem{defn}[equation]{Definition}
\newtheorem{eg}[equation]{Example}
\newtheorem{egs}[equation]{Examples}
\newtheorem{notation}[equation]{Notation}
\newtheorem{exer}[equation]{Exercise}
\newtheorem{exers}[equation]{Exercises}
\newtheorem{recall}[equation]{Recall}

\newtheorem{mainques}[equation]{Main Question}
\newtheorem{assump}[equation]{Assumption}
\newtheorem{term}[equation]{Terminology}
\newtheorem{warn}[equation]{Warning}

\newcommand{\soln}[1]{\medbreak\subsubsection*{\cref{#1}}}
\newcommand{\fullsoln}[2]{\medbreak\subsubsection*{\fullcref{#1}{#2}}}

\crefformat{chapter}{Lecture~#2#1#3}
\crefmultiformat{chapter}{Lectures~#2#1#3}{ and~#2#1#3}{, #2#1#3}{}

\crefname{exer}{exercise}{exercises}
\Crefname{exer}{Exercise}{Exercises}

\crefformat{exer}{Exercise~#2#1#3}
\crefformat{exers}{Exercise~#2#1#3}
\crefformat{egs}{Example~#2#1#3}
\crefformat{unfortunate}{Fact~#2#1#3}

\crefformat{page}{page~#2#1#3}

 \numberwithin{equation}{chapter}

\newcommand{\fullcsee}[2]{(see \cref{#1}\pref{#1-#2})}

\newcommand{\real}{\mathbb{R}}
\newcommand{\integer}{\mathbb{Z}}
\renewcommand{\natural}{\mathbb{N}}
\newcommand{\rational}{\mathbb{Q}}
\newcommand{\complex}{\mathbb{C}}
\newcommand{\hyper}{\mathbb{H}}
\newcommand{\torus}{\mathbb{T}}

\newcommand{\power}[1]{2^{#1}}

\newcommand{\onto}{\mathrel{\lower0.5pt\hbox{\LARGE$\twoheadrightarrow$}}}

\newcommand{\normal}{\triangleleft}
\newcommand{\iso}{\cong}

\newcommand{\Qrank}{\rank_{\rational}}
\newcommand{\Rrank}{\rank_{\real}}

\def\midline#1{\setbox0\hbox{\kern1pt $#1$\kern1pt
}\mathord{\hbox to 0pt{\kern1pt $#1$\hss}\vrule width\wd0
height2.25pt depth -1.75pt}} \def\pp{\midline{p}}

\DeclareMathOperator{\Homeo}{Homeo}

\DeclareMathOperator{\SL}{SL}
\DeclareMathOperator{\PSL}{PSL}
\DeclareMathOperator{\SO}{SO}
\DeclareMathOperator{\SU}{SU}
\DeclareMathOperator{\Sp}{Sp}
\DeclareMathOperator{\Prob}{Prob}

\DeclareMathOperator{\rank}{\mathrm{rank}}
\DeclareMathOperator{\NH}{Near}
\DeclareMathOperator{\QM}{Quasi}

\newcommand{\U}{\overline{U}}
\newcommand{\V}{\underline{V}}

\newcommand{\Hb}{H_b}

\newcommand{\lift}{\widetilde}

\cornersize{0.5}
\newcommand{\heis}[1]{\lower0.75pt\hbox{\ovalbox{#1}}}

\newcommand{\pref}[1]{{\upshape(}\ref{#1}{\upshape)}}
\newcommand{\fullcref}[2]{\cref{#1}{\upshape(}\ref{#1-#2}{\upshape)}}

\makeatletter
\newcommand{\noprelistbreak}{\@nobreaktrue\nopagebreak\smallskip} 
\makeatother

\newcommand{\zz}[1]{\hbox to 0pt{#1\hss}}

\newcommand{\hintit}[1]{\textsf{\smaller{\upshape[}#1{\upshape]}}}
\newcommand{\hint}[1]{\hintit{\emph{Hint:} #1}}

\begin{document}


\frontmatter
\tableofcontents


\mainmatter


 \LectureSeries[Some arithmetic groups that do not act on $S^1$]%
 {Some~arithmetic~groups that do~not~act~on~the~circle \author{Dave Witte Morris}}


\address{Department of Mathematics and Computer Science,
University of Lethbridge, Lethbridge, Alberta, T1K~3M4, Canada}
\email{Dave.Morris@uleth.ca}





\section*{Abstract}
The group $\SL(3,\integer)$ cannot act (faithfully) on the circle (by homeomorphisms).  We will see that many other arithmetic groups also cannot act on the circle.  The discussion will involve several important topics in group theory, such as ordered groups, amenability, bounded generation, and bounded cohomology.

\Cref{LOLect} provides an introduction to the subject, and uses the theory of left-orderable groups to prove that $\SL(3,\integer)$ does not act on the circle.
\Cref{BddGenLect} discusses bounded generation, and proves that groups of the form $\SL \bigl( 2, \integer[\alpha] \bigr)$ do not act on the real line.
\Cref{AmenLect,BddCohoLect} are brief introductions to amenable groups and bounded cohomology, respectively. They also explain how these ideas can be used to prove that actions on the circle have finite orbits.
An appendix provides hints or references for all of the exercises.

These notes are slightly expanded from talks given at the Park City Mathematics Institute's Graduate Summer School in July 2012. The author is grateful to the PCMI staff for their hospitality, the organizers for the invitation to take part in such an excellent conference, and the students for their energetic participation and helpful comments that made the course so rewarding (and improved these notes).


\lecture{Left-orderable groups and a proof for $\SL(3,\integer)$} \label{LOLect}

\section{Introduction}

In Geometric Group Theory (and many other fields of mathematics), one of the main methods for understanding a group is to look at the spaces it can act on. (For example,  speakers at this conference have discussed actions of groups on $\delta$-hyperbolic spaces, $\mathrm{CAT}(0)$ cube complexes, Euclidean buildings, and  other spaces of geometric interest.) 
In these lectures, we consider only very simple spaces, namely, the real line~$\real$ and the circle~$S^1$. 
Also, we consider only a single, very interesting class of groups, namely, the arithmetic groups. More precisely, the topic of these lectures is:

\begin{mainques} \label{MainQues}
Let\/ $\Gamma$ be $\SL(n, \integer)$, or some other arithmetic group. 
	\begin{enumerate}
	\item Does there exist a faithful action of~$\Gamma$ on~$\real$?
	\item Does there exist a faithful action of~$\Gamma$ on~$S^1$?
	\end{enumerate}
All actions are assumed to be continuous, so the questions ask whether there exists a faithful homomorphism $\phi \colon \Gamma \to \Homeo(X)$, where $X = \real$ or~$S^1$. (Recall that a homomorphism is \emph{faithful} if its kernel is trivial.)
\end{mainques}

A fundamental theorem in the subject tells us that the two seemingly different questions in \cref{MainQues} are actually the same for most arithmetic groups (if, as is usual in Geometric Group Theory, we ignore the very minor difference between a group and its finite-index subgroups):

\begin{thm}[Ghys \cite{GhysCercle}, Burger-Monod \cite{BurgerMonod-BddCohoLatts}] \label{GhysFP}
Let\/ $\Gamma = \SL(n,\integer)$, or some other irreducible arithmetic group, such that no finite-index subgroup of\/~$\Gamma$ is isomorphic to a subgroup of\/ $\SL(2,\real)$. Then:
	\begin{align*}
	\text{some } &\text{finite-index subgroup of\/~$\Gamma$ has a faithful action on\/~$\real$}
	\\& \iff \text{ some finite-index subgroup of\/~$\Gamma$ has a faithful action on~$S^1$.}
	\end{align*}
\end{thm}

\begin{proof}
($\Rightarrow$) Suppose $\dot\Gamma$ is a finite-index subgroup of~$\Gamma$ that acts on~$\real$.
Then $\dot\Gamma$ also acts on the one-point compactification of~$\real$, which is homeomorphic to~$S^1$. (Note that this argument is elementary and very general. It is the opposite direction of the theorem that requires assumptions on~$\Gamma$, and sometimes requires passage to a finite-index subgroup.)

($\Leftarrow$)
Suppose $\dot\Gamma$ is a finite-index subgroup of~$\Gamma$ that acts on~$S^1$.
A major theorem proved independently by Ghys \cite{GhysCercle} and Burger-Monod \cite{BurgerMonod-BddCohoLatts} tells us that that the action must have a finite orbit. (We will say a bit about the proof of this theorem in \cref{AmenLect,BddCohoLect}.) This means that a finite-index subgroup $\ddot\Gamma$ of~$\dot\Gamma$ has a fixed point in~$S^1$. 

Let $p$ be a point in~$S^1$ that is fixed by~$\ddot\Gamma$. Then $\{p\}$ is a $\ddot\Gamma$-invariant subset, so its complement is also invariant. This implies that $\ddot\Gamma$ acts on $S^1 \smallsetminus \{p\}$, which is homeomorphic to~$\real$.
\end{proof} 

Thus, in most cases, it does not matter which of the two versions of \cref{MainQues} we consider. For now, let us look at actions on~$\real$.

\begin{assump}
To avoid minor complications, we will assume, henceforth, that 
	$$ \text{all actions are orientation-preserving.} $$
This means that an action of~$\Gamma$ on~$X$ is a faithful homomorphism $\phi  \colon \Gamma  \to \Homeo_+(X)$, where $\Homeo_+(X)$ is the group of \emph{orientation-preserving} homeomorphisms of~$X$. Since $\Homeo_+(X)$ is a subgroup of index~$2$ in the group of all homeomorphisms, this is just another example of ignoring the difference between a group and its finite-index subgroups.
\end{assump}

\begin{rem}
The expository paper \cite{Morris-CanLatt} covers the main topics of these lectures in somewhat more depth.
See \cite{GhysCircleSurvey} and \cite{Navas-GrpsDiffeos} for introductions to the general theory of group actions on the circle (not just actions of arithmetic groups), and see \cite{Morris-IntroArithGrps} for an introduction to arithmetic groups.
\end{rem}

\section{Examples}

The following result provides an obstruction to the existence of an action on~$\real$.

\begin{lem} \label{TorsionNoAct}
If a group has a nontrivial element of finite order, then the group does not have a faithful action on\/~$\real$.
\end{lem}

\begin{proof} 
It suffices to show that every nontrivial element~$\varphi$ of $\Homeo_+(\real)$ has infinite order. 

Since $\varphi$ is nontrivial, there is some $p \in \real$, such that $\varphi(p) \neq p$. Assume, without loss of generality, that $\varphi(p) > p$. The fact that $\varphi$ is an orientation-preserving homeomorphism of~$\real$ implies that it is an increasing function:
	$$ x > y \implies \varphi(x) > \varphi(y) .$$
Therefore (letting $x = \varphi(p)$ and $y = p$), we have $\varphi^2(p) > \varphi(p)$. In fact, by induction, we have
	$$ \varphi^n(p) > \varphi^{n-1}(p) > \cdots > \varphi(p) > p ,$$
so $\varphi^n(p) > p$ for every $n > 0$. This implies $\varphi^n(p) \neq p$, so $\varphi$ is not the identity map. Since $n$~is arbitrary, this means that $\varphi$ has infinite order.
\end{proof}

\begin{cor} \label{SLnZNoActBcsTorsion}
If $n \ge 2$, then $\SL(n,\integer)$ does not have a faithful action on~$\real$. 
\end{cor}

\begin{proof}
It is easy to find a nontrivial element of finite order in $\SL(n,\integer)$. For example, the matrix {\smaller$\begin{bmatrix} -1 & 0 \\ 0 & -1 \end{bmatrix}$} is in $\SL(2,\integer)$ and has order~$2$.
\end{proof}


It is not difficult to show that every arithmetic group has a finite-index subgroup that has no elements of finite order \cite[Lem.~4.19, p.~232]{PlatonovRapinchukBook}. This means that \cref{TorsionNoAct} does not provide any obstruction at all to the existence of actions of sufficiently small finite-index subgroups of~$\Gamma$. For example:

\begin{eg}
It is well known that some finite-index subgroups of $\SL(2,\integer)$ are free groups. (In fact, every torsion-free subgroup is free \cite[Eg.~1.5.3, p.~11, and Prop.~18, p.~36]{Serre-Trees}.) Any such subgroup has \emph{many} faithful actions on~$\real$:
\end{eg}

\begin{exer} \label{FreeGrpActsOnR}
Show that every finitely generated free group has a faithful action on~$\real$.
\end{exer}

Here is a much less trivial class of arithmetic groups that act on~$\real$:

\begin{thm}[{}{Agol and Boyer-Rolfsen-Wiest}] \label{ArithSL2CActs}
If\/ $\Gamma$ is any arithmetic subgroup of $\SL(2,\complex)$, then some finite-index subgroup of\/~$\Gamma$ has a faithful action on\/~$\real$.
\end{thm}

\begin{proof}
A very recent and very important theorem of Agol \cite{AgolHaken} tells us there is a finite-index subgroup~$\dot\Gamma$ of~$\Gamma$, such that there is a surjective homomorphism $\varphi \colon \dot\Gamma  \onto \integer$. Since $\integer$ has an obvious nontrivial action on~$\real$ (by translations), this implies that $\dot\Gamma$ also acts nontrivially on~$\real$ (by translations). However, additional effort is required to obtain an action that is faithful.

A classic theorem of Burns-Hale \cite{BurnsHale} provides a cohomological condition that implies the existence of a faithful action:
	$$ \begin{matrix} 
	\text{$\dot\Gamma$ has a faithful action on~$\real$ if $H^1(\Lambda; \real)$ is nonzero}
	\\ \text{for every finitely generated, nontrivial subgroup~$\Lambda$ of~$\dot\Gamma$}
	\end{matrix} $$
(see \fullcref{LOExers}{BurnsHale}). Agol's theorem tells us $H^1(\dot\Gamma; \real)$ is nonzero, which establishes the hypothesis for the special case where $\Lambda = \dot\Gamma$. By using $3$-manifold topology and a fairly simple argument about Euler characteristics, a theorem of Boyer-Rolfsen-Wiest \cite[Thms.~3.1 and 3.2]{BoyerRolfsenWiest} promotes this nonvanishing to obtain the condition for all~$\Lambda$, and thereby yields a faithful action on~$\real$. 
\end{proof}

\begin{egs} \label{EgNotAct}
We have seen that some finite-index subgroups of $\SL(2,\integer)$ have actions on~$\real$. To obtain arithmetic groups that do \emph{not} act on~$\real$ (even after passing to a finite-index subgroup), we need a bigger group.
	\begin{enumerate}
	\item  \label{EgNotAct-SLn}
	One approach would be to take larger matrices (not just $2 \times 2$). Later in this lecture, we will see that this works: if $n \ge 3$, then no finite-index subgroup of $\SL(n,\integer)$ has a faithful action on~$\real$. 
	\item  \label{EgNotAct-SL2O}
	Another possible approach would be to keep the same size of matrix, but enlarge the ring of coefficients: instead of only the ordinary ring of integers~$\integer$, consider a ring a algebraic integers~$\mathcal{O}$. \Cref{BddGenLect} outlines a proof that this approach also works: if $\alpha$ is a real algebraic integer that is irrational (for example, we could take $\alpha = \sqrt{2}$), then no finite-index subgroup of $\SL \bigl( 2, \integer[\alpha] \bigr)$ acts faithfully on~$\real$.
	\end{enumerate}
\end{egs}

\section{The main conjecture} \label{LOMainConjSect}

In the spirit of \cref{EgNotAct}, it is conjectured that every ``irreducible'' arithmetic group that acts on~$\real$ is contained in a very small Lie group, like $\SL(2,\complex)$:

\begin{conj} \label{NoActConj}
If\/ $\Gamma$ is an ``irreducible'' arithmetic group, then
	$$ \text{$\Gamma$ does not have a faithful action on\/~$\real$} $$
unless $\Gamma$ is an arithmetic subgroup of a ``very small'' Lie group.
\end{conj}

For the interested reader, the remainder of this \namecref{LOMainConjSect} makes the conjecture more precise. However, we will only look at examples of arithmetic groups, not delving deeply into their theory, so, for our purposes, a vague understanding of the conjecture is entirely sufficient.

\begin{defn}
Saying that $\Gamma$ is \emph{irreducible} means that no finite-index subgroup of~$\Gamma$ is a direct product $\Gamma_1 \times \Gamma_2$ (where $\Gamma_1$ and~$\Gamma_2$ are infinite). 
\end{defn}

The following simple observation shows that the problem reduces to this case.

\begin{exer} \label{DirProdFaithful}
Show that the direct product $\Gamma_1 \times \Gamma_2$ has a faithful action on~$\real$ if and only if $\Gamma_1$ and~$\Gamma_2$ both have faithful actions on~$\real$. 
\end{exer}

Technically speaking, instead of saying that the Lie group is ``very small\zz,'' we should say that it is a semisimple Lie group whose ``real rank'' is only~$1$. In other words, up to finite index, it belongs to one of the four following families of groups (up to local isomorphism):
	\begin{itemize}
	\item $\SO(1,n)$ (the isometry group of hyperbolic $n$-space $\hyper^n$),
	or
	\item $\SU(1,n)$ (the isometry group of complex hyperbolic $n$-space),
	or
	\item $\Sp(1,n)$ (the isometry group of quaternionic hyperbolic $n$-space),
	or
	\item $F_{4,1}$ (the isometry group of the hyperbolic plane over the octonions, also known as the  ``Cayley plane").
	\end{itemize}
Since $\SL(2,\complex)$ is locally isomorphic to $\SO(1,3)$, this list does include the examples in \cref{ArithSL2CActs}.

\begin{rem}
\Cref{NoActConj} applies only to actions on~$\real$, not actions on~$S^1$, because some arithmetic groups of large real rank do act on the circle.
Namely, if $G$ is a semisimple Lie group that has $\SL(2,\real)$ as one of its simple factors, then every arithmetic subgroup of~$G$ acts on the circle (by linear-fractional transformations). However, it is conjectured that these are the only such arithmetic groups of large real rank \cite[p.~200]{GhysCercle}.
\end{rem}

\section{Left-invariant total orders}

The following \lcnamecref{ActIffLO} translates \cref{NoActConj} into a purely algebraic question about the existence of a certain structure on the group~$\Gamma$.

\begin{defn}
Let $\Gamma$ be a group.
\noprelistbreak
	\begin{itemize}
	\item A \emph{total order} on a set~$\Omega$ is an transitive, antisymmetric binary relation~$\prec$ on~$\Omega$, such that, for all $a,b \in \Omega$, we have
	$$ \text{either \ $a \prec b$ \ or \ $a \succ b$ \ or \ $a = b$} .$$
	\item When $\prec$ is a total order on a group~$\Gamma$, we can ask that the order structure be compatible with the group multiplication: $\prec$ is \emph{left-invariant} if, for all $a,b,c \in \Gamma$, we have
		$$ a \prec b \iff ca \prec cb .$$
	\end{itemize}
See \cite{KopytovMedvedev} for more about the theory of left-invariant total orders.
\end{defn}

\begin{exer} \label{ActIffLO}
Let $\Gamma$ be a countable group. Then
	$$
	\text{$\Gamma$ has a faithful action on~$\real$
	$\iff$ $\exists$ a left-invariant total order~$\prec$ on~$\Gamma$.} $$
\end{exer}

\begin{proof}[Hint]
 ($\Rightarrow$) If no nontrivial element of~$\Gamma$ fixes~$0$, then we may define
 	$$ a \prec b \iff a(0) < b(0) ,$$
and this is a left-invariant total order. (Recall that each element of~$\Gamma$ acts on~$\real$ via in increasing function, so if $a(0) < b(0)$, then $c \bigl( a(0) \bigr) < c \bigl( b(0) \bigr)$. If $a(0) = b(0)$, the tie can be broken by choosing some other $p \in \real$ and comparing $a(p)$ with $b(p)$.
 
 ($\Leftarrow$) Note that $\Gamma$ acts faithfully (by left translation) by automorphisms of the ordered set $(\Gamma,\prec)$, which is isomorphic (as an ordered set) to a subset of $(\rational,<)$. If it is isomorphic to all of $(\rational,<)$, then $\Gamma$ acts on the Dedekind completion, which is homeomorphic to~$\real$. There is actually no loss of generality in assuming that $(\Gamma,\prec) \iso (\rational,<)$, because $(\Gamma, \prec)$ can be replaced with a left-invariant ordering of $\Gamma \times \rational$ that is order-isomorphic to $(\rational, <)$.
\end{proof}

Therefore, \cref{NoActConj} can be restated as follows:

\begin{conj}[Algebraic version of the conjecture]
 If\/ $\Gamma$ is an irreducible arithmetic group, then
	$$ \text{$\Gamma$ does not have a left-invariant total order} $$
unless $\Gamma$ is an arithmetic subgroup of a ``very small'' Lie group.
\end{conj}

\begin{exer} \label{ProdPos}
Suppose $\prec$ is a left-invariant total order on~$\Gamma$ (and $e$ is the identity element of~$\Gamma$). Show that if $a,b \in \Gamma$ with $a,b \succ e$, then $ab \succ e$ and $a^{-1} \prec e$.
\end{exer}

\section{$\SL(3,\integer)$ does not act on the line}

We can now prove \fullcref{EgNotAct}{SLn}:

\begin{thm}[Witte \cite{Witte-QrankAct1mfld}] \label{SLnZNotLO}
If\/ $\Gamma$ is a finite-index subgroup of\/ $\SL(n,\integer)$, with $n \ge 3$, then there does not exist a left-invariant total order on\/~$\Gamma$.
\end{thm}

The proof is based on understanding a certain famous subgroup~$H$ of $\SL(3,\integer)$:

\begin{notation} \label{HeisNotn}
Let $H$ be the \emph{discrete Heisenberg group}, which means
	$$H = \begin{bmatrix} 1 & \integer & \integer \cr
 		&1 & \integer \\
		 & & 1 \end{bmatrix} 
		 \subset \SL(3,\integer) .$$
For convenience, let us also fix names for some particular elements of~$H$:
	$$ x =  \begin{bmatrix} 
			1 & 1 & 0 \\
 			&1 & 0 \cr
			& & 1\end{bmatrix}
		, \qquad
		y = \begin{bmatrix} 
			1 & 0 & 0 \\
			&1 & 1 \\
 			& & 1\end{bmatrix}
		, \qquad
		z = \begin{bmatrix} 
			1 & 0 & 1 \\
			&1 & 0 \\
			& & 1\end{bmatrix}
		. $$
(Note that $\{ x, y \}$ is a generating set for~$H$.)
\end{notation}

 \begin{exers} \label{HeisExers} \ 
 \noprelistbreak
\begin{enumerate}
	\item \label{HeisExers-z=[x,y]}
	Show $z = [x,y] \in Z(H)$, where 
		\begin{itemize}
		\item $[x,y] = x^{-1} y^{-1} x y$ is the \emph{commutator} of~$x$ and~$y$, 
		and 
		\item$Z(H) = \{\, a \in H \mid ah = ha, \forall h \in H\,\}$ is the \emph{center} of~$H$. 
		\end{itemize}
	\item \label{HeisExers-[xyyk]}
	Show $x^k y^\ell = y^\ell x^k z^{k \ell}$ for $k,\ell \in \integer$.
	\item \label{HeisExers-orderable}
	{\rm(optional)} 
	Show $H$ has a left-invariant total order. 
		\\ \hint{If $N$ is a normal subgroup of~$\Gamma$, such that $N$ and $\Gamma/N$ each have a left-invariant total order, then $\Gamma$ has a left-invariant total order \fullcsee{LOExers}{extension}.}
		\end{enumerate}
 \end{exers}

\begin{notation}
Suppose $\prec$ is a left-invariant total order on a group~$\Gamma$. For $a,b \in \Gamma$, we write $a \ll b$ if $a$ is \emph{infinitely smaller} than~$b$. I.e.,
	$$a \ll b \iff a^n \prec |b|, \ \forall n \in \integer ,
	\quad
	\text{where $|b| = \begin{cases}
	b & \text{if $b \succeq e$} , \\
	b^{-1} & \text{if $b \prec e$}
	. \end{cases}$}
	$$
\end{notation}

Here is the key fact that will be used in the proof:

\begin{lem}[Ault \cite{Ault-RONilpGrps}, Rhemtulla \cite{Rhemtulla-ROGrps}] \label{LOHeis}
If $\prec$ is any left-invariant total order on~$H$, then either $z \ll x$ or $z \ll y$.
 \end{lem}

\begin{proof} 
Assume, for simplicity, that $x,y,z  \succ e$. 
(This actually causes no loss of generality, since there is no harm in replacing some or all of $x$, $y$, and~$z$ by their inverses. This is because we can retain the relation $[x,y] = z$ by interchanging $x$ and~$y$ if necessary, since $[y,x] = z^{-1}$.) From \fullcref{HeisExers}{[xyyk]}, we have
	\begin{align} \label{LOHeisPfQuadratic}
	 y^n x^n y^{-n} x^{-n}   = z^{-n^2} 
	 .\end{align}
Note that the exponent of~$z$ is quadratic in~$n$ (and negative). 

Now suppose $z \not\ll x$ and $z \not\ll y$. Then there exist $p,q \in \integer$, such that $z^p  \succ x$  and  $z^q  \succ y$. 
Therefore  
	$$e \  \prec \  x^{-1}z^p,  \ y^{-1} z^q,  \ x,  \ y ,$$
so, for all $n \in \integer^+$, we have
	\begin{align*}
	e &\prec  y  ^n   \, x^n  \, (y^{-1}z^q)^n  \, (x^{-1} z^p)^n
	&& \text{(\cref{ProdPos})}
	\\& =  y^n  \, x^n  \, y^{-n}  \, x^{-n}  \, z^{qn+pn}
	&& \text{(since $z \in Z(H)$)}
	\\&=  z^{-n^2}  z^{(p+q)n}
	&& \text{\pref{LOHeisPfQuadratic}}
	\\&= z^{\text{(linear)} - \text{(quadratic)}}
	\\&=  \ z^{\text{negative}}
	&& \text{(if $n$ is sufficiently large)}
	\\&\prec e
	&& \text{(since $z \succ e$)}
	. \end{align*}
This is a contradiction.
\end{proof} 

\begin{proof}[Proof of \cref{SLnZNotLO}]
Suppose there is a left-invariant total order on $\Gamma = \SL(3,\integer)$. (For simplicity, we are writing the proof as if $\Gamma$ is the entire group $\SL(3,\integer)$, and leave it as an exercise for the reader to modify the proof to work for finite-index subgroups.)

In \cref{HeisNotn}, we gave names to three particular elements of~$\Gamma$ that have a single off-diagonal~$1$.  For this proof, we actually want to name all six such elements: we call them $\heis1$, $\heis2$, $\heis3$, $\heis4$, $\heis5$, $\heis6$, where
	$$\SL(3,\integer) = \begin{bmatrix}* 
	& \heis 1 & \heis2 \\[2pt]
 	\heis4 & * & \heis3 \\[2pt]
 	\heis5 & \heis6 & *
 	\end{bmatrix} .$$
Thus, for example, $x = \heis1$, $y = \heis3$, and $z= \heis2$, so
$\left\langle \heis1, \heis2, \heis3 \right\rangle$ is the Heisenberg group. 
Actually, there are six copies of the Heisenberg group  in $\Gamma$ (see \fullcref{SL3ZPfExers}{Heis}): 
	\begin{align} \label{SixHeisInSL3}
	\begin{matrix}
	\bigl\langle \heis 1, \heis2, \heis3 \bigl\rangle, & 
	\bigl\langle \heis2, \heis3, \heis4 \bigl\rangle, & 
	\bigl\langle \heis 3, \heis4, \heis5 \bigl\rangle \\[\medskipamount]
	\bigl\langle \heis4, \heis5, \heis6 \bigl\rangle, & 
	\bigl\langle \heis5, \heis6, \heis1 \bigl\rangle, & 
	\bigl\langle \heis 6,\heis 1, \heis2 \bigl\rangle
	. \end{matrix} 
	\end{align}

Since $\bigl\langle \heis1, \heis2, \heis3 \bigl\rangle$ is a Heisenberg group, \cref{LOHeis} tells us that either $\heis2  \ll   \heis1$  or  $\heis2 \ll  \heis3$.
Assume, without loss of generality, that 
	$$\heis2 \ll  \heis3 .$$
Also, since $\bigl\langle \heis2, \heis3, \heis4 \bigl\rangle$ is also a Heisenberg group, \cref{LOHeis} tells us that either $\heis3  \ll  \heis2$  or  $\heis3 \ll  \heis4$. However, we know  $\heis2 \ll  \heis3$, which implies $\heis3  \not\ll  \heis2$. So we must have
	$$\heis3 \ll  \heis4 .$$
Continuing in this way, using the other Heisenberg groups in succession, we have
	$$ \heis 2 \ll \heis3   \ll  \heis4   \ll  \heis5  \ll  \heis6  \ll  \heis1  \ll  \heis2 .$$
By transitivity, this implies $\heis2  \ll \heis2$, which is a contradiction.
\end{proof}

\begin{exers} \label{SL3ZPfExers} \ 
\noprelistbreak
\begin{enumerate}
\item \label{SL3ZPfExers-Heis}
Verify that each of the subgroups listed in \pref{SixHeisInSL3} is isomorphic to the Heisenberg group. More precisely, for $1 \le k \le 6$, verify that there is an isomorphism 
	$\varphi \colon \left\langle \, \ovalbox{$k-1$} \, , \ovalbox{$k$} \, , \ovalbox{$k+1$} \,\right\rangle \to H$, 
such that
	$$ \varphi \left( \ovalbox{$k-1$} \right) = x
	, \quad
	 \varphi \left( \ovalbox{$k$} \right) = z
	 , \text{\quad and}\quad
	 \varphi \left( \ovalbox{$k+1$} \right) = y, 
	  . $$

\item \label{SL3ZPfExers-FinInd}
The given proof of \cref{SLnZNotLO} is incomplete, because it assumes that $\Gamma$ is all of $\SL(3,\integer)$. (And we already knew from \cref{SLnZNoActBcsTorsion} that $\SL(3,\integer)$ has no faithful action on~$\real$.) Modify the proof so it is valid when $\Gamma$ is a finite-index subgroup of $\SL(3,\integer)$.

\end{enumerate}
\end{exers}

\section{Comments on other arithmetic groups}

\begin{rems} \ 
\noprelistbreak
	\begin{enumerate}
	\item The proof of \cref{SLnZNotLO} can be adapted to show that finite-index subgroups of the group $\Sp(4,\integer)$ do not have left-invariant total orders \cite{Witte-QrankAct1mfld}.
	\item From \fullcref{LOExers}{subgrp}, we see that if $\Gamma$ contains a finite-index subgroup of either $\SL(3,\integer)$ or $\Sp(4,\integer)$, then $\Gamma$ does not have a left-invariant total order. In the terminology of arithmetic groups, this exactly means \cite{Witte-QrankAct1mfld}:
		$$ \qquad \text{if $\Qrank \Gamma \ge 2$, then $\Gamma$ does not have a left-invariant total order} .$$
	\item The argument of \cref{SLnZNotLO} relies on the existence of Heisenberg groups in~$\Gamma$, so it does not apply to  $\SL \bigl( 2, \integer[\alpha] \bigr)$. (Heisenberg groups are nonabelian nilpotent groups, but every nilpotent subgroup of $\SL(2,\complex)$ is abelian.) We will use a quite different argument to discuss these groups in \cref{BddGenLect}.
	\item Another important case in which the argument of \cref{SLnZNotLO} cannot be applied is when $G/\Gamma$ is compact. This is because every nilpotent subgroup of $\Gamma$ is virtually abelian. 
	\end{enumerate}
\end{rems}

\begin{open} \label{CocpctOpen}
Find an arithmetic group\/~$\Gamma$,  such that $G/\Gamma$ is compact,  and 
no finite-index subgroup of\/~$\Gamma$ has a faithful action on\/~$\real$.
\end{open}

Most large arithmetic groups have Kazhdan's Property~$(T)$. (We refer the reader to E.\,Breuillard's lectures in this volume for further discussion of property~$(T)$.)
Therefore, a negative answer to the following well-known question would be a major advance toward settling \cref{CocpctOpen} (and many other interesting cases of \cref{NoActConj}):

\begin{open} 
Does there exist an infinite group with Kazhdan's Property\/~$(T)$ that has a faithful action on\/~$\real$ or~$S^1$? 
\end{open}

The answer is negative for actions on the circle if we require our actions to act by homeomorphisms that have continuous second derivatives:

\begin{thm}[Navas \cite{Navas-ActKazhdan}] 
Infinite groups with Kazhdan's Property\/~$(T)$ do not have faithful, $C^2$\!~actions on  $S^1$. 
\end{thm}

\begin{exers} \label{LOExers}
A group that has a left-invariant total order is said to be \emph{left-orderable}.
\noprelistbreak
\begin{enumerate}
\itemsep=\smallskipamount

\item \label{LOExers-subgrp}
Show that every subgroup of a left-orderable group is left-orderable.

\item \label{LOExers-abelian}
Show torsion-free, \emph{abelian} groups are left-orderable.

\item \label{LOExers-extension}
Show that if $N$ is a normal subgroup of~$\Gamma$, such that $N$ and~$\Gamma/N$ are left-orderable, then $\Gamma$ is left-orderable.
\hint{Compare $a$ with~$b$ in $\Gamma/N$, and use the order on~$N$ to break ties.}

\item \label{LOExers-nilpotent}
Show torsion-free, \emph{nilpotent} groups are left-orderable.

\item \label{LOExers-solvable}
{(harder)} Show that some torsion-free, \emph{solvable} group is \emph{not} left-orderable.

\item \label{LOExers-Exps}
Show that $\Gamma$ is left-orderable if and only if, for every finite sequence $g_1,\ldots,g_n$ of nontrivial elements of~$\Gamma$, there exists $\epsilon_1,\ldots,\epsilon_n \in \{\pm1\}$, such that the semigroup generated by $\{g_1^{\epsilon_1},\ldots,g_n^{\epsilon_n}\}$ does not contain~$e$.

\item \label{LOExers-locally}
Show \emph{locally} left-orderable $\implies$ left-orderable.
\\ \hintit{A group is said to \emph{locally} have a certain property if all of its \emph{finitely generated} subgroups have the property. So the exercise asks you to show that if every finitely generated subgroup of~$\Gamma$ is left-orderable, then $\Gamma$ is left-orderable.}

\item \label{LOExers-residually}
Show \emph{residually} left-orderable $\implies$ left-orderable.
\\ \hintit{$\Gamma$ is said to \emph{residually} have a certain property if, for every $g \in \Gamma$, there exists a group~$H$ with the property, and a homomorphism $\varphi \colon \Gamma \to H$, such that $\varphi(g) \neq e$.}

\item \label{LOExers-BurnsHale}
(Burns-Hale \cite[Cor.~2]{BurnsHale}) Show that if $H^1(\Lambda; \real) \neq 0$ for every nontrivial, finitely generated subgroup~$\Lambda$ of~$\Gamma$, then $\Gamma$ is left-orderable.


\end{enumerate}
\end{exers}


\lecture{Bounded generation and a proof for $\SL \bigl( 2, \integer[\alpha] \bigr)$}
\label{BddGenLect}

In \cref{LOLect}, we showed that finite-index subgroups of $\SL(3, \integer)$ do not have faithful actions on~$\real$. In this lecture, we prove the same conclusion for appropriate groups of $2 \times 2$ matrices.

\begin{notation}
Throughout this lecture, $\alpha$ is a algebraic integer that is real and irrational. (Actually, we do not need to require $\alpha$ to be real unless it satisfies a quadratic equation with rational coefficients.)
\end{notation}

\begin{thm}[Lifschitz-Morris \cite{LMActOnLine}] \label{SL2ANoAct}
If\/ $\Gamma$ is a finite-index subgroup of\/ $\SL \bigl( 2, \integer[\alpha] \bigr)$, then\/ $\Gamma$ does not have a faithful action on\/~$\real$.
\end{thm} 

The proof has two ingredients: 
\textit{bounded generation} 
and 
\textit{bounded orbits}. Both are with respect to \emph{unipotent subgroups}.

\begin{notation}
Let $\U =$  {\smaller$ \begin{bmatrix}1 & * \\ 0 & 1\end{bmatrix}$}
and
$\V =$   {\smaller $ \begin{bmatrix} 1 & 0 \\ * & 1 \end{bmatrix}$}.  
These are ``unipotent'' subgroups of $\SL \bigl( 2, \integer[\alpha] \bigr)$.
\end{notation}

\begin{rem}
Any subgroup of $\SL \bigl( 2, {\ast} \bigr)$ that is conjugate to a subgroup of~$\U$ is said to be \emph{unipotent}, but we do not need any unipotent subgroups other than $\U$ and~$\V$.
\end{rem}

\section{What is bounded generation?}  

Now that we know what unipotent subgroups are, let us see what ``bounded generation'' means.

\begin{recall}
A basic theorem of undergraduate linear algebra says that every invertible matrix can be reduced to the identity matrix by row operations (or by column operations, if you prefer those). Also, since performing a row operation (or column operation) is the same as multiplying by an ``elementary matrix\zz,'' this implies the important fact that every invertible matrix is a product of elementary matrices. In other words, 
	$$ \text{\it the elementary matrices generate the group of all invertible matrices.} $$
\end{recall}

However, in your undergraduate course, the scalars were assumed to be in a \emph{field} (probably either $\real$ or~$\complex$), but our matrices have their entries in a ring of integers (namely, either $\integer$ or $\integer[\alpha]$), which is not a field. Fortunately, this is not a problem:

\begin{eg} \label{RowReduceEg}
The matrix 
	{\smaller$\begin{bmatrix} 13 & 31 \\ 5 & 12 \\ \end{bmatrix}$}
is a fairly typical element of $\SL(2,\integer)$. Let us see that it can be reduced to the identity matrix, by using only \emph{integer} row operations. More precisely, the only allowable operation is adding an integer ($\integer$) multiple of one row to another row. (In linear algebra, a few additional operations are usually allowed, such as multiplying a row by a scalar, but we will not permit those operations.) Using $\leadsto$ to denote applying a row operation, we see that the matrix can indeed be reduced to the identity:
$$\begin{bmatrix}
13 & 31 \\
5 & 12 \\
\end{bmatrix}
\leadsto
\begin{bmatrix}
3 & 7 \\
5 & 12 \\
\end{bmatrix}
\leadsto
\begin{bmatrix}
3 & 7 \\
2 & 5 \\
\end{bmatrix}
\leadsto
\begin{bmatrix}
1 & 2 \\
2 & 5 \\
\end{bmatrix}
\leadsto
\begin{bmatrix}
1 & 2 \\
0 & 1 \\
\end{bmatrix}
\leadsto
\begin{bmatrix}
1 & 0 \\
0 & 1 \\
\end{bmatrix} . $$
\end{eg}

Here is the general case:

\begin{prop} \label{RowReduceSL2Z}
Every matrix in $\SL\bigl(2,\integer)$ can be reduced to the identity matrix by integer row operations.
\end{prop}

\begin{proof}[Idea of proof]
Essentially, we apply the Euclidean Algorithm: 
	\begin{itemize}
	\item Choose the smallest nonzero entry in the first column, and use a row operation to subtract an appropriate integer multiple of it from the other entry in the first column. Now it is the other entry that is the smallest in the first column.
	\item By repeating this process, we will eventually reach a situation with only one nonzero entry in the first column. Since the determinant is~$1$, this entry must be a unit in the ring~$\integer$, which means that the entry is either $1$ or~$-1$. By performing just a few more row operations, we can assume it is~$1$, and that it is in the top-left corner.
	\item Now the matrix is upper triangular, with a $1$ in the top-left corner. Since the determinant is~$1$, there must also be a~$1$ in the bottom right corner. So one additional row operation yields the identity matrix.
	\qedhere \end{itemize}
\end{proof}

Here is another way of saying the same thing:

\begin{cor} \label{UVGenSL2Z}
$\U$ and~$\V$ generate $\SL(2,\integer)$. 
\end{cor}

\begin{proof}
Adding an integer multiple of one row to another row is the same as multiplying on the left by a matrix in either $\U$ or~$\V$.
\end{proof}

In \cref{RowReduceEg}, we used only a few (namely, $5$) row operations to reduce the matrix to the identity by using the procedure outlined in the proof of \cref{RowReduceSL2Z}. However, it is easy to construct examples of matrices for which this procedure will use an arbitrarily large number of steps. It would be much better to have a more clever algorithm that can reduce every matrix to the identity in no more than, say, $1000$ row operations.
Unfortunately, this is \emph{impossible:}

\begin{unfortunate}[see \fullcref{BddGenExers}{SL2Z}] \label{ReduceSL2ZUnbdd}
For every $c$, there is a matrix in $\SL(2,\integer)$ that cannot be reduced to the identity with less than~$c$ integer row operations.
\end{unfortunate}

Here is another way of saying this:  \cref{UVGenSL2Z} tells us that every element of $\SL(2,\integer)$ can be written as a word in the elements of~$\U$ and~$\V$. What \cref{ReduceSL2ZUnbdd} tells us is that there is no uniform bound on the length of the word: some elements of $\SL(n,\integer)$ require a word of length more than a hundred, others require length more than a million, others require length more than a trillion, and so on. 

In other words, if $g \in \SL(2,\integer)$, then \cref{UVGenSL2Z} tells us, for some~$n$, there are sequences $\{u_i\}_{i=1}^n \subset \U$ and $\{v_i\}_{i=1}^n \subset \V$, such that 
	$$ g = u_1 v_1 u_2 v_2 \cdots u_n v_n .$$
However, \cref{ReduceSL2ZUnbdd} tells us that there is no uniform bound on~$n$ that is independent of~$g$. That is, although $\U$ and~$\V$ \emph{generate} $\SL(2,\integer)$, they do not \emph{boundedly} generate $\SL(2,\integer)$.

Here is the official definition:

\begin{defn}
Suppose $X_1,\ldots,X_k$ are subgroups of~$\Gamma$. We say $X_1,\ldots,X_k$ \emph{boundedly generate}~$\Gamma$ if there is some~$n$, such that
	$$ \Gamma = (X_1 X_2\cdots,X_k)^n .$$
In other words, for every $g \in \Gamma$, there exist sequences $\{x_{i,j}\}_{j=1}^n \subseteq X_i$, such that
	$$ g = x_{1,1} x_{2,1} \cdots x_{k,1} \, x_{1,2} x_{2,2} \cdots x_{k,2} \, \cdots \, x_{1,n} x_{2,n} \cdots x_{k,n} .$$ 
The key point (which is what distinguishes this from just saying the subgroups \emph{generate}~$\Gamma$) is that there is an upper bound on~$n$ that is independent of~$g$. (Then, since $x_{i,j}$ is allowed to be the identity element, we can take the same value of~$n$ for all~$g$.)
\end{defn}

\begin{exers} \label{BddGenExers}
When $\Gamma$ is boundedly generated by cyclic subgroups (i.e., $\Gamma = H_1 H_2 \cdots H_n$, with each $H_i$ cyclic), we usually just say $\Gamma$ is \emph{boundedly generated}.
\noprelistbreak
\begin{enumerate} 

\item \label{BddGenExers-quotient}
Assume $\Gamma$ is boundedly generated (by cyclic subgroups), and $N$ is a normal subgroup of~$\Gamma$. Show that $\Gamma/N$ is boundedly generated (by cyclic subgroups).

\item \label{BddGenExers-modnth}
Assume $\Gamma$ is boundedly generated (by cyclic subgroups), and $n \in \integer^+$.
 Show $\bigl\langle\, g^n \mid g \in \Gamma \,\bigr\rangle$ has finite index in~$\Gamma$.

\item \label{BddGenExers-FinInd}
Let $\dot\Gamma$ be a finite-index subgroup of~$\Gamma$. Show that $\Gamma$ is boundedly generated (by cyclic subgroups) if and only if $\dot\Gamma$~is boundedly generated (by cyclic subgroups)

\item \label{BddGenExers-SL2Z}
Prove \cref{ReduceSL2ZUnbdd}.
\hint{The free group~$F_2$ is not boundedly generated by cyclic subgroups \fullcsee{QuasiMExers}{F2NotBddGen}. You may assume this fact (without proof).}

\item \label{BddGenExers-VariablePowers}
(harder) Assume $\Gamma$ is boundedly generated (by cyclic subgroups), and $n$ is any function from~$\Gamma$ to~$\integer^+$.
 Show $\bigl\langle\, g^{n(g)} \mid g \in \Gamma \,\bigr\rangle$ has finite index in~$\Gamma$.

\end{enumerate}
\end{exers}

\section{Bounded generation of $\SL \bigl( 2, \integer[\alpha] \bigr)$} \label{BddGenSL2ASect}

Although there is no bound on the number of $\integer$ operations needed to reduce a $2 \times 2$ matrix to the identity, we can find a bound if we allow slightly more scalars in our operations:

\begin{thm}[Carter-Keller-Paige \cite{CKP,Morris-CKP}] \label{SL2OBddGen}
The subgroups $\U$ and~$\V$ boundedly generate\/ $\SL\bigl(2,\integer  {[\alpha]} \bigr)$. 
\end{thm}

In other words, there is some~$n$, such that 
	$$ \SL\bigl(2,\integer  {[\alpha]} \bigr)
	= (\U \V)^n
	=  \U \, \V \, \U \, \V \, \cdots\, \U \, \V .$$

\begin{rem}
The proof of the general case of \cref{SL2OBddGen} is nonconstructive, so it does not provide an explicit bound on~$n$. In cases where a bound is known, it depends on~$\alpha$
and can be arbitrarily large if the algebraic integer $\alpha$ is very complicated. However, it is believed that there should be a uniform bound that is independent of~$\alpha$. In fact, we will see below that if certain number-theoretic conjectures are true, then less than 10 row operations should always suffice.
\end{rem}

The known proofs of \cref{SL2OBddGen} are long and complicated, so we will not try to explain them. Instead, we will give a very short and simple proof that relies on an unproved conjecture in Number Theory.

\begin{defn}
Let $r$ and~$q$ be nonzero integers. We say $r$ is a \emph{primitive root} modulo~$q$ if 
	$$\{\,  r,   r^2,  r^3,   \ldots \,\} \mod q   \ = \  \{1,2,3, \ldots, q-1\} .$$
\end{defn}

\begin{eg}
$3$ is a primitive root modulo~$7$, because
	$$ 3, \ 3^2 \equiv 2, \ 3^3 \equiv 6, \ 3^4 \equiv 4, \ 3^5 \equiv 5, \ 3^6 \equiv 1 $$
is a list of all the nonzero residues modulo~$7$.
\end{eg}

Although it is still an open problem, we will assume:

\begin{conj}[Artin's Conjecture]
Let $r \in \integer$, such that $r \neq -1$ and $r$ is not a perfect square.
Then there exist infinitely many primes~$q$, such that $r$ is a primitive root modulo~$q$.
\end{conj}

\begin{rem}
Although this conjecture is still an open problem, it is implied by a certain generalization of the Riemann Hypothesis \cite{Hooley-ArtinConj}.
\end{rem}

Dirichlet proved there are infinitely many primes in any (appropriate) arithmetic progression \cite[p.~61]{Serre-CourseArith}, and we will assume a stronger form of Artin's Conjecture that says $q$ can be chosen to be in any such arithmetic progression:

\begin{conj} \label{ArtinArithProgConj}
Let $a,b,r \in \integer$, such that 
	\begin{itemize}
	\item $\gcd(a,b) = 1$, 
	and
	\item $r \neq -1$ and $r$~is not a perfect $m$th power for any $m > 1$.
	\end{itemize}
Then there exist infinitely many primes~$q$ in the arithmetic progression $\{a + k b\}_{k \in \integer}$, such that $r$ is a primitive root modulo~$q$.
\end{conj}

\begin{rem}
We assume $r$ is not a perfect $m$th power in order to avoid the following obstruction:
if $m > 1$, $a \equiv 1 \pmod{m}$, $m \mid b$, and $r$~is a perfect $m$th power, then $r$ is not a primitive root modulo any prime in the arithmetic progression $\{a + k b\}_{k = 0}^\infty$.
\end{rem}

To avoid the need for any Algebraic Number Theory, we will prove a slight variation of \cref{SL2OBddGen} that replaces the algebraic number~$\alpha$ with a rational number, namely, $1/p$:

\begin{thm}[Carter-Keller-Paige \cite{CKP,Morris-CKP}]
If $p$ is any prime, then every matrix in $\SL \bigl( 2, \integer[1/p] \bigr)$ can be reduced to the identity matrix with a bounded number of $\integer[1/p]$ row operations.
\end{thm}

\begin{proof} 
Let $\begin{bmatrix} a &c \\ b & d \end{bmatrix} \in \SL \bigl( 2, \integer[1/p] \bigr)$. Assume, for simplicity, that $a,b,c,d \in \integer$. 
We explain how to reduce this matrix to the identity with only five row operations:
	$$ \begin{bmatrix} a & c \\ b & d \end{bmatrix}
	\stackrel{\textstyle 1}{\leadsto}
	\begin{bmatrix} q & * \\ b & d \end{bmatrix}
	\stackrel{\textstyle 2}{\leadsto}
	\begin{bmatrix} q & * \\ p^\ell & * \end{bmatrix}
	\stackrel{\textstyle 3}{\leadsto}
	\begin{bmatrix} 1 & * \\ p^\ell & * \end{bmatrix}
	\stackrel{\textstyle 4}{\leadsto}
	\begin{bmatrix} 1 & * \\ 0 & * \end{bmatrix}
	\stackrel{\textstyle 5}{\leadsto}
	\begin{bmatrix} 1 & 0 \\ 0 & 1 \end{bmatrix}
	.$$

1) Letting $r = p$ in \cref{ArtinArithProgConj}, we know there is some $k \in \integer$, such that $q = a + kb$ is prime, and $p$~is a primitive root modulo~$q$. Our first row operation adds $k$~times the second row to the first row.

2) Now, since $p$ is a primitive root modulo~$q$, we know there exists $\ell \in \integer^+$, such that $p^\ell \equiv  b \pmod{q}$. This means there exists $k' \in \integer$, such that $p^\ell = b  + k' q$. Our second row operation adds $k'$~times the first row to the second row.
 
3) Every element of $\integer[1/p]$ is a multiple of $p^\ell$ (since $p$ is a unit in this ring). Therefore, our third row operation can add anything at all to the top-left entry, so we can change this matrix entry to anything we want. We change it to a~$1$ (by subtracting $(q-1)p^{-\ell}$ times the second row from the first row).

4) Now, the fourth row operation can use the~$1$ in the top-left corner to kill the bottom-left entry by subtracting $p^\ell$ times the first row from the second row.

5) Since the original matrix {\smaller[2]$\begin{bmatrix} a &c \\ b & d \end{bmatrix}$} is in $\SL \bigl( 2, \integer[1/p] \bigr)$, and the row operations we applied do not affect the determinant, we know that the determinant of this matrix is~$1$. Therefore, the bottom-right entry must be~$1$. So the fifth row operation can kill the top-right entry. 
\end{proof}

\begin{rem}[Carter-Keller \cite{CarterKeller-BddElemGen,CarterKeller-ElemExp}] \label{SL3ZBddGen}
If $n \ge 3$, then every matrix in $\SL(n,\integer)$ can be reduced to the identity by using no more than $\frac{1}{2}(3n^2 - n) + 36$ integer row operations.
Thus, although the Euclidean Algorithm will use an unbounded number of row operations, Carter and Keller showed that the extra freedom provided by larger matrices can be exploited to find a different algorithm that uses only a bounded number.
\end{rem}


\section{Bounded orbits and a proof for $\SL \bigl( 2 , \integer[\alpha] \bigr)$}

We have seen that the subgroups $\U$ and~$\V$ boundedly generate $\SL \bigl( 2 , \integer[\alpha] \bigr)$. The other ingredient in our proof of \cref{SL2ANoAct} is that these subgroups have bounded orbits:

\begin{thm}[Lifschitz-Morris \cite{LMActOnLine}] \label{BddOrbs}
Let\/ $\Gamma$ be a finite-index subgroup of\/ $\SL \bigl( 2 , \integer[\alpha] \bigr)$. If\/ $\Gamma$ acts on\/~$\real$, then every $\U$-orbit is a bounded set, and every $\V$-orbit is a bounded set.
\end{thm}

Before discussing the proof of this theorem, let us explain how it is used: 

\begin{proof}[Proof of \cref{SL2ANoAct}]
Suppose $\Gamma$ has a nontrivial action on~$\real$. (This will lead to a contradiction.) Pretend, for simplicity, that $\Gamma$ is all of $\SL \bigl( 2 , \integer[\alpha] \bigr)$, instead of being a finite-index subgroup.

\setcounter{step}{0}

\begin{step} \label{AssumeNoFP}
We may assume the action has no fixed points.
\end{step}
Let $F$ be the set of fixed points of the $\Gamma$-action on~$\real$. Then $F$ is obviously a closed set, so its complement is open. (Also, the complement is nonempty, because the $\Gamma$-action is nontrivial.) Thus, if we let $I$ be a connected component of the complement, then $I$ is an open interval in~$\real$. It is easy to see that $I$ is $\Gamma$-invariant (because, by definition, the endpoints of~$I$ are fixed points for~$\Gamma$). So $\Gamma$ acts on~$I$. 

By definition, $I$ is contained in the \emph{complement} of the set of fixed points, so the $\Gamma$-action on~$I$ has no fixed points. Since $I$ is homeomorphic to~$\real$, this provides an action of~$\Gamma$ on~$\real$ with no fixed points.

\begin{step} \label{OrbitsBdd}
We show that every $\Gamma$-orbit on~$\real$ is bounded.
\end{step}
Fix some $p \in \real$. 
	\begin{itemize}
	\item From \cref{BddOrbs}, we know that the $\V$-orbit of~$p$ is bounded, so it has a finite infimum~$a_1$ and a finite supremum~$b_1$. Thus, $\V p$ is contained in the compact interval $[a_1,b_1]$.
	\item From \cref{BddOrbs}, we know that $\U$-orbit of~$a_1$ has a finite infimum~$a_2$, and the $\U$-orbit of~$b_1$ has a finite supremum~$b_2$. Since every element of~$\Gamma$ acts via an order-preserving homeomorphism of~$\real$, we know that $\U\V p \subseteq [a_2,b_2]$.
	\item Continuing in this way, we see that $(\U\V)^n p$ is contained in a compact interval $[a_{2n}, b_{2n}]$ for every $n \in \integer^+$.
	\end{itemize}
However, since $\U$ and~$\V$ boundedly generate~$\Gamma$ (see \cref{SL2OBddGen}), we know there is some~$n$, such that $(\U\V)^n = \Gamma$. Therefore, the $\Gamma$-orbit of~$p$ is a bounded set.

\begin{step}
We obtain a contradiction.
\end{step}
Fix some $p \in \real$. From \cref{OrbitsBdd}, we know $\Gamma p$ is a bounded set, so it has a finite supremum~$b$. Since $\Gamma p$ is a $\Gamma$-invariant set, and $b$~is a point that is defined from this set, we know that the point~$b$ must be fixed by~$\Gamma$. This contradicts \cref{AssumeNoFP}, which tells us that there are no fixed points.
\end{proof}

Instead of actually proving \cref{BddOrbs}, we will prove a simpler version that replaces $\alpha$ with the rational number $1/p$ (much as in \cref{BddGenSL2ASect}):

\begin{thm}[Lifschitz-Morris \cite{LMActOnLine}] 
Let\/ $\Gamma = \SL \bigl( 2 , \integer[1/p] \bigr)$ {\upshape(}or a finite-index subgroup{\upshape)}, where $p$~is prime. If\/ $\Gamma$ acts on\/~$\real$, then every $\U$-orbit is a bounded set, and every $\V$-orbit is a bounded set.
\end{thm}

\begin{proof}[Idea of proof]
Pretend, for simplicity, that $\Gamma$ is all of $\SL \bigl( 2 , \integer[1/p] \bigr)$, instead of being a finite-index subgroup.
For $u,v \in \integer[1/p]$, let
	$$\overline u = \begin{bmatrix}1 & u \\ 0 & 1 \end{bmatrix} \in \U
	\text{\quad and\quad}
	\underline v =  \begin{bmatrix}1 & 0 \\ v & 1\end{bmatrix} \in \V
	. $$
Also let
	$$ \pp = \begin{bmatrix}p & 0 \\ 0 & 1/p\end{bmatrix} \in \Gamma .$$
A simple calculation shows
	\begin{align} \label{CommRelns}
	\text{$\pp^{n} \overline{u} \pp^{-n} \to \overline{\infty}$
	\quad and\quad
	$\pp^{n} \underline{v} \pp^{-n} \to \underline{0}$
	\qquad
	as $n \to \infty$}
	. \end{align}

Suppose some $\U$-orbit is not a bounded set. Then either it is not bounded above, or it is not bounded below. Assume, without loss of generality, that it is not bounded above. Since $\V$ is conjugate to~$\U$, this implies that some $\V$-orbit is also not bounded above. 
In fact, it can be shown that there is a single point $x \in \real$, such that 
	\begin{itemize}
	\item both the $\U$-orbit and the $\V$-orbit of~$x$ are not bounded above,
	and
	\item $\pp$ fixes~$x$.
	\end{itemize}

Fix $\overline u \in \U$ with $\overline u (x) > x$. Since the $\V$-orbit of~$x$ is not bounded above, we can choose $\underline v \in \V$ with $\overline u (x) < \underline v(x)$. Then, since $\pp$ is order-preserving, we have
	\begin{align*}
	\pp^n \bigl( \overline{u} (x) \bigr) < \pp^n \bigl( \underline{v}(x) \bigr)
	. \end{align*}
However, as $n \to \infty$, we have
	\begin{align*}
	\pp^n \bigl( \overline{u} (x) \bigr) 
	&=  (\pp^n \overline{u} \pp^{-n})(x)
	&& \text{($\pp$ fixes $x$)}
	\\&\to \overline{\infty}(x) 
	&& \text{(\ref{CommRelns})}
	\\&\to \infty  
	&& \text{($\U$-orbit is not bounded above)}
	\\
\intertext{and}
	\pp^n \bigl( \underline{v} (x) \bigr) 
	&=  (\pp^n \underline{v} \pp^{-n})(x)
	\to   \underline{0}(x) 
	. \end{align*}
Therefore $\infty$ is less than the finite number~$\underline{0}(x)$. This is a contradiction.
\end{proof}

\begin{exer} \label{BddGenIsom}
 Assume $\Gamma$ is boundedly generated (by cyclic subgroups).
Show that if $\Gamma$ acts by \emph{isometries} on a metric space~$X$,
and every cyclic subgroup has a bounded orbit on~$X$,
 then every $\Gamma$-orbit on~$X$ is bounded.
\end{exer}

\section{Implications for other arithmetic groups of higher rank} 

\Cref{SLnZNotLO,SL2ANoAct} each provide examples of arithmetic groups that cannot act on the line. In both cases, the proof was based on unipotent elements, which means that they only apply to arithmetic groups that are \emph{not} cocompact (cf. \cref{CocpctOpen}). To complete the treatment of such groups, it will suffice to consider only a few more examples:

\begin{thm}[Chernousov-Lifschitz-Morris \cite{ChernousovLifschitzMorris-AlmMin}] 
If\/ $\Gamma$ is any noncocompact, irreducible arithmetic group, and\/ $\Rrank \Gamma > 1$, then\/ $\Gamma$ contains a finite-index subgroup of either\/ $\SL \bigr(2, \integer[\alpha] \bigr)$ {\upshape(}for some~$\alpha${\upshape)} or a noncocompact arithmetic subgroup of either\/ $\SL(3,\real)$ or\/ $\SL(3,\complex)$. 
\end{thm}

Combining this with \cref{SL2ANoAct} establishes the following observation:

\begin{cor}
Proving the following very special case would establish \cref{NoActConj} under the additional assumption that\/ $\Gamma$ is not cocompact.
\end{cor}

\begin{conj} \label{LattSL3NoAct}
Noncocompact arithmetic subgroups of\/ $\SL(3,\real)$ and\/ $\SL(3,\complex)$ 
have no faithful action on\/~$\real$.
\end{conj}

One possible approach is to use bounded generation:

\begin{thm}[Lifschitz-Morris \cite{LMActOnLine}]
Let\/ $\Gamma$ be a noncocompact arithmetic subgroup of\/ $\SL(3,\real)$ or\/ $\SL(3,\complex)$. If some finite-index subgroup of\/~$\Gamma$ is boundedly generated by unipotent subgroups, then\/ $\Gamma$ does not have a faithful action on\/~$\real$.
\end{thm}

This implies that \cref{LattSL3NoAct} is a consequence of the following fundamental conjecture in the theory of arithmetic groups:

\begin{conj}[Rapinchuk, 1989] \label{RapinchukConj}
If\/ $\Gamma$ is a noncocompact, irreducible arithmetic group, and\/ $\Rrank \Gamma > 1$, then\/ $\Gamma$ contains a finite-index subgroup that is boundedly generated by unipotent subgroups.
\end{conj}

In fact, to establish \cref{LattSL3NoAct} (and therefore also the entire noncocompact case of \cref{NoActConj}), it would suffice to prove the special case of \cref{RapinchukConj} in which $\Gamma$ is an arithmetic subgroup of either $\SL(3,\real)$ or $\SL(3,\complex)$.


\lecture{What is an amenable group?} \label{AmenLect}

Amenability is a very fundamental notion  in group theory --- there are literally dozens of different definitions that single out exactly the same class of groups. We will discuss just a few of these many viewpoints, and, for simplicity, we will restrict our attention to \emph{discrete} groups that are countable, ignoring the important applications of this notion in the theory of topological groups. Much more information can be found in the monographs \cite{PatersonBook} and~\cite{PierBook}.

\section{Ponzi schemes}

Let us begin with an amusing example that illustrates one of the many definitions.
 
\begin{eg} \label{PonziFreeGrp}
Consider the free group  $F_2  = \langle a,b \rangle$, and let us assume that every element of the group starts with \$1. Thus, if $f_0(g)$ denotes the amount of money possessed by element~$g$ at time $t = 0$, then
	 $$ \text{$f_0(g)  = \$1$,   \quad for all~$g \in F_2$.} $$
Now, everyone will pass their dollar to the person next to them who is 
closer to the identity. (That is, if $g = x_1x_2\cdots x_n$ is a reduced word, with $x_i \in \{a^{\pm1}, b^{\pm1} \}$ for each~$i$, then $g$~passes its dollar to $g' = x_1x_2\cdots x_{n-1}$. The identity element has nowhere to pass its dollar, so it keeps the money it started with.) Then, letting $f_1$ denote the amount of money possessed now (at time $t = 1$), we have
	$$ \text{$f_1(g) =  \$3$  \quad for all~$g$ \  (except that $f_1(e) = \$5$)}. $$
Thus, everyone has more than doubled their money. Furthermore, this result was achieved by moving the money only a bounded distance. 
\end{eg} 

Such an arrangement is called a \emph{Ponzi scheme} on the group~$F_2$:

\begin{defn}
A \emph{Ponzi scheme} on a group~$\Gamma$ is a function $M \colon \Gamma \to \Gamma$, such that:
	\begin{enumerate}
	\item $M^{-1}(g) \ge 2$ for all $g \in \Gamma$
		$$ \text{(\emph{everyone doubles their money if each $g$ passes its dollar to $M(g)$}),} $$
		and
	\item there is a finite subset~$S$ of~$\Gamma$, such that $M(g) \in gS$ for all $g \in \Gamma$
		$$ \text{(\emph{money moves only a bounded distance}).} $$
	\end{enumerate}
\end{defn}

From \cref{PonziFreeGrp}, we know there is a Ponzi scheme on the free group~$F_2$. However, not all groups have a Ponzi scheme:

\begin{exer} \label{AbelNoPonzi}
There does not exist a Ponzi scheme on the abelian group~$\integer^n$.
\\ \hint{Any group with a Ponzi scheme must have exponential growth, because $f_t(g)$ is exponentially large, but the money moves only a linear distance.}
\end{exer}

More generally, we will see later that no solvable group has a Ponzi scheme (even though solvable groups can have exponential growth). This is because solvable groups are ``amenable\zz,'' and the nonexistence of a Ponzi scheme can be taken as the definition of amenability:

\begin{thm}[{}{Gromov \cite[p.~328]{Gromov-MetricStructures}}] 
\label{PonziIffNotAmenThm}
There exists a Ponzi scheme on\/~$\Gamma$ if and only if\/ $\Gamma$~is not amenable.
\end{thm}

This theorem provides a nice description of what it means for a group to \emph{not} be amenable, but it does not directly provide any positive information about a group that \emph{is} amenable. Most of the other definitions we discuss are better for that.

\section{Almost-invariant subsets}

Instead of using Ponzi schemes, we will adopt the following definition:
	$$ \text{$\Gamma$ is amenable  $\iff$   $\Gamma$ has \emph{almost-invariant} finite subsets.} $$
To see what this means, let us consider an example:

\begin{eg}  
Let $\Gamma = \integer^2 = \langle a, b \rangle$, where $a = (1,0)$ and $b = (0,1)$.

If we let $F$ be a large ball in~$\Gamma$, then $F$ is very close to being invariant under the left-translation by~$a$ and~$b$:
	$$ \text{
	$\#(F \cap aF)  >  (1-\epsilon) \, \#F$
	\ and \ 
	$\#(F \cap bF)  >  (1-\epsilon) \, \#F$
	} , $$
where $\epsilon$ can be as small as we like, if we take~$F$ to be sufficiently large.

We say that $F$ is ``almost invariant:''
\end{eg}

\begin{defn}
Let $\Gamma$ be a group, and fix a finite subset~$S$ of~$\Gamma$ and some $\epsilon > 0$. A finite, nonempty subset~$F$ of~$\Gamma$ is \emph{almost invariant} if
	$$ \text{$\#(F \cap aF)  > (1-\epsilon) \, \#F$,   \quad $\forall a  \in S$} .$$
\end{defn}

\begin{defn} \label{AmenDefn}
$\Gamma$ is \emph{amenable} if and only if $\Gamma$ has  almost-invariant  finite subsets (for all finite~$S$ and all $\epsilon > 0$).
\end{defn}

\begin{exers} \label{FolnerEx}
Use \cref{AmenDefn} to show:
	\begin{enumerate}

	\item \label{FolnerEx-Free}
The free group  $F_2$  is \emph{not} amenable. 
 \hint{If $F$ is almost invariant, then the first letter of most of the words in~$F$ must be \emph{both~$a$ and~$b$}.}
 
	\item \label{FolnerEx-union}
If $\Gamma$ is amenable, $S$ is a finite subset of~$\Gamma$, and $\epsilon > 0$, then there exists a finite subset~$F$ of~$\Gamma$, such that $\#(SF) < (1 + \epsilon) \, \#F$, where 
		$$SF = \{\, sf \mid s \in S, f \in F\,\} .$$

	\item \label{FolnerEx-QI}
Amenability is invariant under quasi-isometry. (This means that you should assume $\Gamma_1$ is quasi-isometric to~$\Gamma_2$, and prove that $\Gamma_1$ is amenable if and only if $\Gamma_2$ is amenable.)
 \hint{Fix $c > 1$. Show $\Gamma$ is not amenable iff it has a finite subset~$S$, such that $\#(SF) \ge c \cdot \#F$ for every finite subset~$F$ of~$\Gamma$.}
 
	\item \label{FolnerEx-noPonzi}
Amenable groups do not have Ponzi schemes.
	
	\item \label{FolnerEx-Ponzi}
Nonamenable groups have Ponzi schemes.
	\hint{Suppose $A_1,\ldots,A_n$ are finite sets, and $a_1,\ldots,a_n \in \natural$. Show (by induction on $a_1 + \cdots + a_n$) that if $\# \bigcup_{i \in I} A_i \ge \sum_{i \in I} a_i$ for every $I \subseteq \{1,\ldots,n\}$, then there exists $A_i' \subseteq A_i$, such that $\#A_i' = a_i$ and $A_1',\ldots,A_n'$ are pairwise disjoint.}

	\end{enumerate}
\end{exers}

\begin{term} \ 
\noprelistbreak
\begin{enumerate}
\item Since the notion of ``almost invariant" depends on the choice of $S$ and~$\epsilon$, many authors say that $F$ is ``$(S,\epsilon)$-invariant\zz.''

\item An almost-invariant set can also be called a ``F\o lner set\zz.'' More precisely,
a sequence $\{F_n\}$ of nonempty, finite subsets of~$\Gamma$ is said to be a \emph{F\o lner sequence} if, for every finite subset~$S$ of~$\Gamma$, every $\epsilon > 0$, and every sufficiently large~$n$, the set~$F_n$ is $(S,\epsilon)$-invariant. (The set~$F_n$ is often called a ``F\o lner set\zz.'') Thus, \cref{AmenDefn} can be restated as saying that $\Gamma$ is amenable if and only if it has a F\o lner sequence.
\end{enumerate}
\end{term}

\section{Average values and invariant measures} \label{AvgValSect}

In many situations, it is difficult to directly employ the almost-invariant sets provided by \cref{AmenDefn}. This \lcnamecref{AvgValSect} provides some consequences that are often easier to apply. For example, every bounded function on~$\Gamma$ has an average value:

\begin{defn}
A \emph{mean} on $\ell^\infty(\Gamma)$ is a linear functional $A \colon \ell^\infty(\Gamma) \to \complex$, such that $A(\varphi)$ satisfies two axioms that would be expected of the average value of~$\varphi$:
	\begin{itemize}
	\item \emph{the average value of a constant function is that constant}
		$$\text{$A(c) =  c$ \quad if $c$ is a constant} ,$$
	and
	\item \emph{the average value of a positive-valued function cannot be negative}
		$$ \text{$A(\varphi) \ge 0$ \quad if $\varphi \ge 0$} .$$
	\end{itemize}
The mean is \emph{left-invariant} if $A \bigl( \varphi^g \bigr) =  A \bigl( \varphi \bigr)$ for all $\varphi \in \ell^\infty(\Gamma)$ and all $g \in \Gamma$, where $\varphi^g(x) = \varphi(gx)$.
\end{defn}

\begin{prop} \label{AmenAvgVal}
$\Gamma$ is amenable if and only if there exists a left-invariant mean on $\ell^\infty(\Gamma)$.
\end{prop}

\begin{proof}[Proof ($\Rightarrow$)]
Choose a sequence~$\{F_n\}_{n=1}^\infty$ of almost-invariant sets with $\epsilon \to 0$ as $n \to \infty$, and let
	$$A_n(\varphi) =  \frac{1}{\#F_n}\sum_{x \in F_n} \varphi(x) .$$
That is, $A_n(\varphi)$ is the average value of~$\varphi$ on the finite set~$F_n$, so $A_n$ is obviously a mean on $\ell^\infty(\Gamma)$.

Since the set $F_n$ is almost invariant, the mean~$A_n$ is close to being left-invariant. To obtain perfect left-invariance, we take a limit:
	$$ A(\varphi) = \lim_{k \to \infty} A_{n_k}(\varphi) ,$$
where $\{n_k\}$ is a subsequence chosen so that the limit exists.

However, if we choose different subsequences for different functions~$\varphi$, then the limit may not be linear or left-invariant --- we need to be consistent in our choice of $A(\varphi)$  for all~$\varphi$. 
This can be accomplished in various ways:
	\begin{itemize}
	\item (logician's approach) An \emph{ultrafilter} on~$\natural$ tells us which subsequences are ``good'' and which are ``bad\zz.'' So the choice of an ultrafilter easily leads to a consistent value  for $A(\varphi)$.
	\item (analyst's approach) Define a linear functional~$A_0$ that 
		\begin{itemize}
		\item takes the value~$1$ on the constant function~$1$, 
		and
		\item is $0$ on every function of the form $\varphi^g - \varphi$.
		\end{itemize}
		Then the \emph{Hahn-Banach Theorem} tells us that $A_0$ extends to a linear functional defined on all of~$\ell^\infty(\Gamma)$. 
	\item (other viewpoints) Use \emph{Zorn's Lemma}, \emph{Tychonoff's Theorem}, or some other version of the \emph{Axiom of Choice}.
	\qedhere
	\end{itemize}
\end{proof} 

The following consequence is very important in the theory of group actions:

\begin{cor} \label{ActHasInvtMeas}
Suppose
\noprelistbreak
	\begin{itemize}
	\item $\Gamma$ is amenable, 
	and
	\item $\Gamma$ acts on a compact metric space~$X$  {\upshape(}by homeomorphisms\/{\upshape)}.
	\end{itemize}
Then there exists a $\Gamma$-invariant Borel probability measure~$\mu$ on~$X$.
\end{cor}

\begin{proof}
Fix a basepoint $x_0 \in X$. For any $f \in C(X)$, we can define a function $\varphi_f \colon \Gamma \to \complex$ by restricting~$f$ to the $\Gamma$-orbit of~$x_0$. More precisely,
	$$ \varphi_f(g) = \varphi( gx_0) .$$
Since $f$ is continuous and~$X$ is compact, we know that $f$ is bounded. So $\varphi_f$ is also bounded. Therefore, \cref{AmenAvgVal} tells us that it has an average value $A(\varphi_f)$, which we call $\mu(f)$.

Since $A$ is a mean, it is easy to see that $\mu$ is a positive linear functional of finite norm (in fact, $\|\mu\| = 1$). So the Riesz Representation Theorem tells us that $\mu$ is a Borel measure on~$X$. Since $A$ is translation-invariant and $A(1) = 1$, we see that $\mu$ is translation-invariant and $\mu(X) = 1$, so $\mu$ is a translation-invariant probability measure.
\end{proof}

\begin{rem}
The converse is true: if every $\Gamma$-action on every compact metric space has a $\Gamma$-invariant probability measure, then $\Gamma$ is amenable. So this is another possible choice for the definition of amenability.
\end{rem}

\Cref{BddCohoLect} will discuss the ``bounded cohomology group'' $\Hb^n(\Gamma;V)$, which is defined exactly like the usual group cohomology, except that  all cochains are required to be  bounded functions. This notion provides another definition of amenability:

\begin{thm}[B.\,E.\,Johnson \cite{Johnson-CohoBanach}] 
$\Gamma$ is amenable if and only if $\Hb^n(\Gamma;V) =  0$ for every\/ $\Gamma$-module~$V$ that is the dual of a Banach space.
\end{thm}

\begin{proof}[Proof  of $(\Rightarrow)$]  Recall that if $\Gamma$ is a finite group, and $V$~is a $\Gamma$-module (such that multiplication by the scalar $|\Gamma|$ is invertible), then one can prove $H^n(\Gamma;V) =  0$  by averaging: for an $n$-cocycle $\alpha \colon \Gamma^n \to V$, define
 	$$\overline{\alpha}(g_1,\ldots,g_{n-1})  = \frac{1}{|\Gamma|} \sum_{g \in \Gamma}  \alpha(g_1,\ldots,g_{n-1}, g) .$$
Then $\overline{\alpha}$ is an $(n-1)$-cochain, and $\delta \overline{\alpha} = \pm\alpha$. So $\alpha$ is a coboundary, and is therefore trivial in cohomology.

Since $\Gamma$ is  amenable,  we can do exactly this kind of averaging  for any  \emph{bounded} cocycle. See \cref{Hbb(amen;R)=0} for more details.
\end{proof}

When $\Gamma$ is amenable, \cref{AmenAvgVal} allows us to take the average value of the characteristic function of any subset of~$\Gamma$. This leads to von\,Neumann's original definition of amenability \cite{vonNeumann-AllgemeineMasses}:

\begin{cor} \label{VonNeumann}
$\Gamma$ is amenable if and only if there exists a finitely additive, translation-invariant probability measure that is defined on all of the subsets of\/~$\Gamma$.
\end{cor}

More precisely, if we let $\power{\Gamma}$ be the collection of all subsets of~$\Gamma$, then the conclusion means there is a function $\mu \colon \power{\Gamma} \to [0,1]$, such that:
	\begin{itemize}
	\item $\mu(X_1 \cup X_2) = \mu(X_1) + \mu(X_2)$ if $X_1$ and~$X_2$ are disjoint,
	\item $\mu(\Gamma) = 1$,
	and
	\item $\mu(gX) = \mu(X)$ for all $g \in \Gamma$ and $X \subseteq \Gamma$.
	\end{itemize}

This definition was motivated by von\,Neumann's interest in the famous \emph{Banach-Tarski paradox}. The subjects are connected via the following notion:

\begin{defn}
A \emph{paradoxical decomposition} of~$\Gamma$ is a representation
	$$ \Gamma = \left( \coprod_{i=1}^m A_i \right) \coprod \left( \coprod_{j=1}^n B_j \right) \qquad \text{(disjoint unions)} ,$$
such that, for some $g_1,\ldots,g_m,h_1,\ldots,h_n \in \Gamma$, we have
	$$\Gamma = \bigcup_{i=1}^m g_i A_i =  \bigcup_{j=1}^n h_j B_j .$$
\end{defn}

\begin{exers} \label{VonNeumannEx} \ 
\noprelistbreak
\begin{enumerate}
\item \label{VonNeumannEx-noParadox}
Show that if $\Gamma$ is amenable, then $\Gamma$ does not have a paradoxical decomposition.

\item \label{VonNeumannEx-free}
Find an explicit paradoxical decomposition of a free group.

\item \label{VonNeumannEx-Paradox}
Show that if $\Gamma$ is not amenable, then $\Gamma$ has a paradoxical decomposition.
\\ \hint{There exists a Ponzi scheme.} 

\end{enumerate}
\end{exers}

\begin{rems} \ 
\noprelistbreak
	\begin{enumerate}
	\item von\,Neumann used the German word ``messbar'' (which can be translated as ``measurable''), not the currently accepted term ``amenable\zz,'' and his condition was not proved to be equivalent to \cref{AmenDefn} until much later (by F\o lner \cite{Folner-GrpsBanachMean}).
	\item See \cite{Wagon-BanachTarski} for much more about the Banach-Tarski paradox, paradoxical decompositions, and the relevance of amenability.
	\end{enumerate}
\end{rems}

We have now seen several proofs that the existence of almost-invariant sets implies some other notion that is equivalent to amenability. Here a proof that goes the other way.

\begin{prop}
If there is an invariant mean~$A$ on $\ell^\infty(\Gamma)$, then $\Gamma$ has almost-invariant finite sets.
\end{prop}

\begin{proof}[Idea of proof]
The dual of~$\ell^1(\Gamma)$ is~$\ell^\infty(\Gamma)$, so $\ell^1(\Gamma)$ is dense in the dual of $\ell^\infty(\Gamma)$, in an appropriate weak topology. Hence, there is a sequence $\{f_n\} \subset \ell^1(\Gamma)$, such that $f_n \to A$. Since $A$~is invariant, we may choose some large~$n$ so that $f_n$ is close to being invariant. Then, for an appropriate $c> 0$, the finite set $\{\, x \mid |f(x)| > c \,\}$ is almost invariant.
\end{proof}

\section{Examples of amenable groups}

Much of the following exercise can be proved fairly directly by using almost-invariant sets, but it will be much easier to use other characterizations of amenability for some of the parts.

\begin{exers} \label{EgAmenEx}
Show that all groups of the following types are amenable:
	\begin{enumerate}

	\item \label{EgAmenEx-finite}
	finite groups

	\item \label{EgAmenEx-cyclic}
	cyclic groups

	\item \label{EgAmenEx-product}
	$\text{amenable} \times \text{amenable}$
	\hintit{I.e., if $\Gamma_1$ and$\Gamma_2$ are amenable, then $\Gamma_1 \times \Gamma_1$ is amenable.}

	\item \label{EgAmenEx-abelian}
	abelian groups

	\item \label{EgAmenEx-extension}
	amenable by amenable
	\qquad \hintit{I.e., if there is a normal subgroup~$N$ of~$\Gamma$, such that $N$ and $\Gamma/N$ are amenable, then  $\Gamma$ is amenable.}

	\item \label{EgAmenEx-solv}
	solvable groups

	\item \label{EgAmenEx-subgrp}
	subgroups of amenable groups

	\item \label{EgAmenEx-quotient}
	quotients of amenable groups

	\item \label{EgAmenEx-locally}
	locally amenable groups
	\qquad \hintit{I.e., if every \emph{finitely generated} subgroup of~$\Gamma$ is amenable, then  $\Gamma$ is amenable.}

	\item \label{EgAmenEx-limit}
	direct limits of amenable groups
	\qquad \hintit{I.e., if $\mathcal{A}$ is a collection of amenable groups that is totally ordered under inclusion, then $\bigcup \mathcal{A}$ is amenable.}

	\item \label{EgAmenEx-subexp}
	groups of subexponential growth
	\qquad \hintit{I.e., if there is a finite generating set~$S$ of~$\Gamma$, such that $\lim_{n \to \infty} (\# S^n)/e^{\epsilon n} = 0$ for every $\epsilon > 0$, then $\Gamma$ is amenable.}

	\end{enumerate}
\end{exers}

\begin{rems} \label{AmenClasses} \ 
\noprelistbreak
\begin{enumerate}

\item \label{AmenClasses-elem}
From \cref{EgAmenEx}, we see that any group obtained from finite groups and abelian groups by repeatedly taking extensions, subgroups, quotients, and direct limits must be amenable. These ``obvious'' examples of amenable groups are said to be \emph{elementary amenable}. 

\item \label{AmenClasses-Grigorchuk}
The so-called ``Grigorchuk group'' is an example of a group with subexponential growth that is not elementary amenable \cite{GrigorchukPak-Intermediate}.

\item \label{AmenClasses-Basilica}
A group is said to be \emph{subexponentially amenable} if it can be constructed from groups of subexponential growth by repeated application of extensions, subgroups, quotients, and direct limits. The ``Basilica group'' is an example of an amenable group that is not subexponentially amenable \cite{BartholdiVirag-AmenRandomWalks}.

\end{enumerate}
Thus, the following obvious inclusions are proper:
$$ \begin{matrix}
		\hfill \{\text{finite}\} \subsetneq \\[2pt]
		\{\text{abelian}\} \subsetneq\{\text{solvable}\} \subsetneq 
		\end{matrix}
	\left\{ \begin{matrix} \text{elementary} \\ \text{amenable} \end{matrix} \right\}
	\subsetneq \left\{ \begin{matrix} \text{subexponentially} \\ \text{amenable} \end{matrix} \right\}
	\subsetneq \left\{ \begin{matrix} \text{amenable} \end{matrix} \right\}
	. $$
\end{rems}
	
\begin{rem}
By combining \fullcref{FolnerEx}{Free} with \fullcref{EgAmenEx}{subgrp}, we see that if $\Gamma$ has a nonabelian free subgroup, then $\Gamma$ is not amenable. The converse is often called the ``von\,Neumann Conjecture\zz,'' but it was shown to be false in 1980 when Ol'shanskii proved that the ``Tarski monster'' is not amenable. This is a group in which every element has finite order, so it certainly does not contain free subgroups \cite{OlshanskiiSapir-TorsionByCycic}.
N.\,Monod \cite{Monod-PiecewiseProj} has recently constructed counterexamples that are much less complicated.
\end{rem}

\begin{warn}
In the theory of topological groups, it is \emph{not} true that every subgroup of an amenable group is amenable --- only the \emph{closed} subgroups need to be amenable. In particular, many amenable topological groups contain nonabelian free subgroups (but such subgroups cannot be closed).
\end{warn}

\section{Applications to actions on the circle}

We are discussing amenability in these lectures because it plays a key role in the proofs of Ghys \cite{GhysCercle} and Burger-Monod \cite{BurgerMonod-BddCohoLatts} that large arithmetic groups must always have a finite orbit when they act on the circle (cf.\ \cref{GhysFP}). Here is a much simpler example of the connection between amenability and finite orbits:

\begin{prop} \label{AmenOnCircle}
Suppose
	\begin{itemize}
	\item $\Gamma$ is amenable,
	and
	\item $\Gamma$ acts on $S^1$ {\upshape(}by orientation-preserving homeomorphisms\/{\upshape)}.
	\end{itemize}
Then either
	\begin{enumerate}
	\item the abelianization of\/~$\Gamma$ is infinite,
	or
	\item \label{AmenOnCircle-FinOrb}
	the action has a finite orbit.
	\end{enumerate}
\end{prop}

\begin{proof}
From \cref{ActHasInvtMeas}, we know there is a $\Gamma$-invariant probability measure~$\mu$ on~$S^1$.

\setcounter{case}{0}

\begin{case}
Assume $\mu$ has an atom.
\end{case}
This assumption means there exists some point $p \in S^1$ that has positive measure: $\mu \bigl( \{p\} \bigr) > 0$. Since $\mu$ is $\Gamma$-invariant, every point in the orbit of~$p$ must have the same measure. However, since $\mu$ is a probability measure, we know that the sum of the measures of these points is finite. Therefore the orbit of~$p$ must be finite (since the sum of infinitely many copies of the same positive number is infinite).

\begin{case}
Assume $\mu$ has no atoms.
\end{case}
Assume, for simplicity, that the support of~$\mu$ is all of~$S^1$. (That is, no nonempty open interval has measure~$0$.) Then the assumption of this case implies that, after a continuous change of coordinates, the measure~$\mu$ is simply the Lebesgue measure on~$S^1$. Since $\Gamma$ preserves this measure (and is orientation-preserving), this implies that $\Gamma$ acts on the circle by rotations. Since the group of rotations is abelian, we conclude that the abelianization of~$\Gamma$ is infinite. (Or else the image of~$\Gamma$ in the rotation group is finite, which means that every orbit is finite.)
\end{proof}

\begin{rem}
If we assume that $\Gamma$ is infinite and finitely generated, then the conclusion of the \lcnamecref{AmenOnCircle} can be strengthened: it can be shown that the abelianization of~$\Gamma$ is infinite \cite{Morris-AmenOnLine}, so there is no need for alternative~\pref{AmenOnCircle-FinOrb}.
\end{rem}

Large arithmetic groups always have finite abelianization, so it might seem that the theorem of Ghys and Burger-Monod could be obtained directly from \cref{AmenOnCircle}. Unfortunately, that is not possible, because arithmetic groups are not amenable (since they contain free subgroups). Instead, Ghys's proof is based on the following more sophisticated observations:

\begin{prop} \label{AmenOnConvex}
Suppose
	\begin{itemize}
	\item $\Gamma$ is amenable, 
	\item $\Gamma$ acts by  {\upshape(}continuous{\upshape)} linear maps on a locally convex vector space~$V$,
	and 
	\item $C$ is a nonempty, compact, convex, $\Gamma$-invariant subset of~$V$.
	\end{itemize}
Then $\Gamma$ has a fixed point  in~$C$. 
\end{prop}

\begin{proof}
$\Gamma$ acts on the compact set~$C$ by homeomorphisms, so \cref{ActHasInvtMeas} provides a $\Gamma$-invariant probability measure~$\mu$ on~$C$. Let $p$ be the center of mass of~$\mu$. Then $p$ is fixed by~$\Gamma$, since $\mu$ is $\Gamma$-invariant.  Also, since $C$ is convex, we know $p \in C$.
\end{proof}

\begin{rem}
\Cref{AmenOnConvex} has a converse: if $\Gamma$ has a fixed point in every nonempty, compact, convex, $\Gamma$-invariant set, then $\Gamma$ is amenable. So this fixed-point property provides yet another possible definition of amenability.
\end{rem}

\begin{cor}[{}{Furstenberg}] \label{FurstenbergLemma}
Suppose 
	\begin{itemize}
	\item $\Gamma$ is an arithmetic subgroup of\/ $\SL(3,\real)  = G$, 
	\item $\Gamma$ acts on~$S^1$ {\upshape(}by homeomorphisms{\upshape)},
	\item $P = \text{\smaller[2] $\begin{bmatrix} * & * & * \\ & * & * \\ & & * \end{bmatrix}$} \subset G$,
	and
	\item $\Prob(S^1) = \{ \text{probability measures on~$S^1$} \}$, with the natural weak topology.
	\end{itemize}
Then there exists a\/ $\Gamma$-equivariant measurable function $\overline\psi  \colon G/P  \to  \Prob(S^1)$.
\end{cor}

\begin{proof}
Let
	$$ \mathcal{C} = \bigl\{\, \text{measurable $\Gamma$-equivariant}\ \psi \colon G \to \Prob(S^1) \,\bigr\} $$
(where functions that differ only on a set of measure~$0$ are identified).
It is easy to see that $\mathcal{C}$ is convex. Then, since the Banach-Alaoglu Theorem tells us that weak$^*$-closed, convex, bounded sets are compact, we see that $\mathcal{C}$ is compact in an appropriate weak topology. Also, $P$ acts continuously on~$\mathcal{C}$, via $\psi^p(g) = \psi(gp)$.

We know that solvable groups are amenable \fullcsee{EgAmenEx}{solv}. Although we have only been considering discrete groups, the same is true for topological groups in general. So $P$ is amenable (because it is solvable). Therefore (a generalization of) \cref{AmenOnConvex} tells us that $P$ has a fixed point. This means there is a $\Gamma$-equivariant map $\psi \colon G \to \Prob(S_1)$, such that $\psi(gp) = \psi(g)$ (a.e.). Ignoring a minor issue about sets of measure~$0$, this implies that $\psi$ factors through to a well-defined $\Gamma$-equivariant function $\overline\psi \colon G/P \to \Prob(S_1)$.
\end{proof} 

We omit the proof of the main step in Ghys's argument:

\begin{thm}[Ghys \cite{GhysCercle}] \label{GhysConstant}
The function~$\overline\psi$ provided by \cref{FurstenbergLemma} is constant {\upshape(}a.e.{\upshape)}.
\end{thm}

From this, it is easy to complete the proof:

\begin{cor}[Ghys \cite{GhysCercle}]
If\/ $\Gamma$ is any arithmetic subgroup of\/ $\SL(3,\real)$, then every action of\/~$\Gamma$ on the circle has a finite orbit.
\end{cor}

\begin{proof}
From \cref{GhysConstant}, we know there is a constant function $\overline\psi \colon G/P \to \Prob(S^1)$ that is $\Gamma$-equivariant (a.e.). 
	\begin{itemize}
	\item Since $\overline\psi$ is constant, its range is a single point~$\mu$ (a.e.). 
	\item Since $\overline\psi$ is $\Gamma$-equivariant, its range is a $\Gamma$-invariant set.
	\end{itemize}
So $\mu$ is $\Gamma$-invariant. Since $\mu \in \Prob(S^1)$, then the proof of \cref{AmenOnCircle} shows that either 
	\begin{enumerate}
	\item the abelianization of~$\Gamma$ is infinite,
	or
	\item the action has a finite orbit.
	\end{enumerate}
Since the abelianization of every arithmetic subgroup of $\SL(3,\real)$ is finite, we conclude that there is a finite orbit, as desired.
\end{proof}

\begin{rem}
See \cite{GhysCircleSurvey} for a nice exposition of Ghys's proof for the special case of lattices in $\SL(n,\real)$. (A slightly modified proof of the general case that reduces the amount of case-by-case analysis is in \cite{WitteZimmer-ActOnCircle}.) A quite different (and very interesting) version of the proof is in \cite{BaderFurmanShaker}.
\end{rem}


\lecture{Introduction to bounded cohomology} \label{BddCohoLect}

M.\,Burger and N.\,Monod \cite{BurgerMonod-BddCohoLatts,BurgerMonod-ContBddCoho} developed a sophisticated machinery to calculate bounded cohomology groups, and used it to prove that actions of arithmetic groups on the circle have finite orbits. (See \cite{Monod-ContBddCoho} for an exposition.) We will discuss only some elementary aspects of bounded cohomology, and describe how it is related to actions on the circle, without explaining the fundamental contributions of Burger-Monod. 
See \cite{Monod-Invitation} for a more comprehensive introduction to bounded cohomology and its applications. (Almost all of the information in this lecture can be found there.)
The widespread interest in this subject was inspired by a paper of Gromov \cite{Gromov-VolBddCoho}.

\section{Definition}

\begin{recall}
For a discrete group~$\Gamma$, the cohomology group $H^n(\Gamma; \real)$ is defined as follows.
	\begin{itemize}
	\item Any function $c \colon \Gamma^n \to \real$ is an \emph{$n$-cochain}, and the set of these cochains is denoted $C^n(\Gamma)$.
	\item A certain \emph{coboundary operator} $\delta_n \colon C^n(\Gamma) \to C^{n+1}(\Gamma)$ is defined. Here are the definitions for the smallest values of~$n$:
		\begin{align*}
		\delta_0 c \, (g_1) &= 0 && \text{for $c \in \real$}, \\
		\delta_1c \, (g_1,g_2) &= c(g_1 g_2) - c(g_1) - c(g_2) && \text{for $c \colon \Gamma \to \real$}.
		\end{align*}
\item Then
	$$H^n(\Gamma; \real) 
	= \frac{\ker \delta_n}{\mathop{\mathrm{Image}} \delta_{n-1}} 
	=  \frac{\text{$n$-cocycles}}{\text{$n$-coboundaries}} 
	=  \frac{Z^n(\Gamma)}{B^n(\Gamma)} .$$
	\end{itemize}
(Note that, for simplicity, we take the coefficients to be~$\real$, not a general $\Gamma$-module.)
\end{recall}

\begin{defn}
The \emph{bounded cohomology} group $\Hb^n(\Gamma;\real)$ is defined in exactly the same way as $H^n(\Gamma;\real)$, except that all cochains are required to be \emph{bounded} functions.
\end{defn}

\begin{eg}
$\Hb^0(\Gamma;\real)$ and $\Hb^1(\Gamma;\real)$ are very easy to compute:
\begin{itemize}
\item It is easy to check that 
	$$H^0(\Gamma;\real) = \{\, \text{$\Gamma$-invariants in $\real$} \,\} = \real = \{\, \text{the set of constants} \,\} .$$
\item The same calculation shows that $\Hb^0(\Gamma)$ is the set of \emph{bounded} constants. Then, since it is obvious that every constant is a bounded function, we have 
	$\Hb^0(\Gamma;\real) = \real$.
\item It is easy to check that $H^1(\Gamma;\real) = \{\, \text{homomorphisms $\Gamma \to \real$} \,\}$.
\item The same calculation shows that $\Hb^1(\Gamma;\real)$ is the set of \emph{bounded} homomorphisms into~$\real$. Since a homomorphism into~$\real$ can never be bounded (unless it is trivial), this means $\Hb^1(\Gamma;\real) = \{0\}$.
\end{itemize}
Thus, $\Hb^0(\Gamma;\real)$ and $\Hb^1(\Gamma;\real)$ give no information at all about~$\Gamma$. (One of them is always~$\real$, and the other is always~$\{0\}$.)
\end{eg}

So $\Hb^n(\Gamma;\real)$ is only interesting when $n \ge 2$. These groups are not easy to calculate:

\begin{eg} \label{BddCoho(F2)}
For the free group~$F_2$, we have:
	$$ \Hb^n(F_2;\real) = \begin{cases}
	\text{$\infty$-dimensional} & n = 2,3 \\
	\hfil \langle\text{\it open problem}\,\rangle & n > 3.
	\end{cases}
	$$
\end{eg}

\begin{open}
Find some countable group\/~$\Gamma$, such that you can calculate $\Hb^n(\Gamma;\real)$ for all~$n$ {\upshape(}and $\Hb^n(\Gamma;\real) \neq 0$ for some~$n${\upshape)}.
\end{open}

Bounded cohomology is easy to calculate for amenable groups:

\begin{prop}[B.\,E.\,Johnson \cite{Johnson-CohoBanach}] \label{Hbb(amen;R)=0}
If\/ $\Gamma$ is amenable, then $\Hb^n(\Gamma;\real) = 0$ for all~$n$.
\end{prop}

\begin{proof}
From \cref{AmenAvgVal}, we know there is a left-invariant mean 
	$$A \colon \ell^\infty(\Gamma; \real) \to \real .$$

Any element of $\Hb^n(\Gamma;\real)$ is represented by a bounded function $c \colon \Gamma^n \to \real$, such that $\delta_n c = 0$.
To simplify the notation, let us assume $n = 2$. For each $g \in \Gamma$, we can define a bounded function $c_g \colon \Gamma \to \real$ by $c_g(x) = c(g,x)$. Then, by defining
	$ \overline c(g) = A(c_g) \in \real $,
we have $\overline c \colon \Gamma \to \real$.

Now, for $g_1,g_2,x \in \Gamma$, we have
	$$ 0
	= \delta_2 c \, (g_1,g_2, x)
	= c(g_1,g_2) - c(g_1, g_2 x) + c(g_1g_2, x) - c(g_2, x) .$$
Applying~$A$ to both sides (considered as functions of~$x$), and recalling that $A$ is left-invariant, we obtain
	$$ 0 = c(g_1,g_2) - \overline{c}(g_1) + \overline{c}(g_1g_2) - \overline{c}(g_2) ,$$
so $c = -\delta_1 \overline{c} \in \mathrm{Image}(\delta_1)$. Therefore $[c] = 0$ in $\Hb^2(\Gamma;\real)$.
\end{proof}

\begin{rems} \ 
\noprelistbreak
	\begin{enumerate}
	\item The bounded cohomology $\Hb^n(X;\real)$ of a topological space~$X$ is defined by stipulating that a cochain in $C^n(X)$ is \emph{bounded} if it is a bounded function on the space of singular $n$-simplices. 
	\item (Brooks \cite{Brooks-RemBddCoho}, Gromov \cite{Gromov-VolBddCoho}) $\Hb^n(X;\real) = \Hb^n \bigl( \pi_1(X) ; \real \bigr)$.
	\item Forgetting that the cochains are bounded yields a comparison homomorphism $\Hb^n(\Gamma;\real) \to H^n(\Gamma;\real)$. It is very interesting to find situations in which this map is an isomorphism.
	\item (Thurston)
	If $M$ is a closed manifold of negative curvature, then the comparison map $\Hb^n(M;\real) \to H^n(M;\real)$ is \emph{surjective} for $n\ge 2$. However, it can fail to be injective.
	\end{enumerate}
\end{rems}

\section{Application to actions on the circle}

\begin{defn}
If $\rho \colon \Gamma \to \Homeo^+(S^1)$ is a homomorphism, then, for each $g \in \Gamma$, covering-space theory tells us that $\rho(g)$ can be lifted to a homeomorphism~$\lift g$ of the universal cover, which is~$\real$. However, the lift depends on the choice of a basepoint, so it is not unique --- lifts can differ by an element of $\pi_1(S^1) = \integer$. Specifically, if $\widehat g$ is another lift of~$\rho(g)$, then
	$$\exists n \in \integer, \ \forall t \in \real, \ \widehat g(t) =  \lift g(t)  + n . $$
Therefore, for any $g,h \in \Gamma$, there exists $c(g,h) \in \integer$, such that
	$$ \forall t \in \real, \ \lift g \bigl( \lift h(t) \bigr) =  \lift{gh}(t)  + c(g,h) , $$
because $\lift g \lift h$ and $\lift{gh}$ are two lifts of~$gh$.

It is easy to verify that $c$ is a $2$-cocycle:
	$$ \text{$c(h,k)  - c(gh,k)  + c(g,hk)  - c(g,h)  = 0$ for $g,h,k \in \Gamma$} ,$$
and that choosing a different lift~$\lift g$ only changes~$c$ by a coboundary. Therefore, $c$ determines a well-defined cohomology class $\alpha\in H^2(\Gamma;\integer)$, which is called the \emph{Euler class} of~$\rho$. 
\end{defn}

\begin{exer} \label{EulerTrivial}
Show that the Euler class of~$\rho$ is trivial if and only if $\rho$ lifts to a homomorphism $\lift\rho \colon \Gamma \to \Homeo^+(\real)$.
 \end{exer}
 
 \begin{rem}
  The Euler class can also be defined more naturally, by noting that if we let $\widetilde H$ be the set consisting of all possible lifts of all elements of $\Homeo^+(S^1)$, then we have a short exact sequence
	 $$ \{e\} \to \integer \to \widetilde H \to \Homeo^+(S^1) \to \{e\} $$
with $\integer$ in the center of~$\widetilde H$. Any such central extension is determined by a well-defined cohomology class $\alpha_0 \in H^2 \bigl( \Homeo^+(S^1);\integer \bigr)$, and the Euler class is obtained using the homomorphism~$\rho$ to pull this class back to~$\Gamma$.
\end{rem}
 
 \begin{exer} \label{EulerIsBdd}
Choose a basepoint in~$\real$ (say, $0$), and assume the lift $\lift g$ is chosen with $0 \le \lift g(0) < 1$ for all $g \in \Gamma$. Show $c$ is bounded. 
\end{exer}

\begin{defn}
Although we only defined bounded cohomology with real coefficients, the same definition can be applied with~$\integer$ in place of~$\real$. Therefore, if we choose $c$ as in \cref{EulerIsBdd}, then it represents a bounded cohomology class $[c]  \in \Hb^2(\Gamma; \integer)$, which is called the \emph{bounded Euler class} of the action.
\end{defn}

\begin{rem}
It can be shown that the bounded Euler class is a well-defined invariant of the action (independent of the choice of basepoint, etc.).
\end{rem}

\begin{prop}[Ghys \cite{Ghys-GrpsEtBddCoho}] \label{BddEulerClassFP}
The bounded Euler class is trivial if and only if\/ $\Gamma$ has a fixed point in~$S^1$.
\end{prop}

\begin{proof}
($\Leftarrow$) We may assume the fixed point is the basepoint~$\overline{0} \in S^1$.
Then we may choose $\lift g$ with $\lift g(0) = 0$. So $c(g,h) = 0$ for all $g,h$.

($\Rightarrow$) We have $c(g,h) = \varphi(gh) - \varphi(g) - \varphi(h)$ for some bounded $\varphi \colon \Gamma \to \integer$.
Letting $\widehat g (t) = \lift g(t) + \varphi(g)$, we have 
	\begin{itemize}
	\item $\widehat g \, \widehat h  = \widehat{gh}$, so $\widehat \Gamma$ is a lift of~$\Gamma$ to $\Homeo^+(\real)$,
	and
	\item $|\widehat g(0)| \le |\lift g(0)| + |\varphi(g)| \le 1 + \|\varphi\|_\infty$.
	\end{itemize}
Hence, the $\widehat\Gamma$-orbit of~$0$ is a bounded set, so it has a supremum in~$\real$. 
This supremum is a fixed point of~$\widehat\Gamma$, so its image in~$S^1$ is a fixed point of~$\Gamma$.
\end{proof}

\begin{cor}
If $\Hb^2(\Gamma;\integer) = 0$, then every orientation-preserving action of~$\Gamma$ on~$S^1$ has a fixed point. \qed
\end{cor}

The following result is easier to apply, because it uses real coefficients for the cohomology, instead of integers:

\begin{cor} \label{H2RVanishFP}
If
	\begin{itemize}
	\item $\Hb^2(\Gamma;\real) = 0$,
	\item $H^1(\Gamma;\real) = 0$,
	and
	\item $\Gamma$ is finitely generated,
	\end{itemize}
then every orientation-preserving action of\/~$\Gamma$ on~$S^1$ has a finite orbit.
\end{cor}

\begin{proof}
The short exact sequence 
	$$ 0 \to \integer \to \real \to \torus \to 0$$
yields a long exact sequence of bounded cohomology:
	$$ \Hb^1(\Gamma;\torus) \to \Hb^2(\Gamma;\integer) \to \Hb^2(\Gamma;\real) .$$
By assumption, the group at the right end is~$0$, so the map on the left is surjective. Therefore, the bounded Euler class is the coboundary of some (bounded) $1$-cocycle $\alpha \colon \Gamma \to \torus$. I.e., $\alpha$ is a homomorphism to $\torus$. Since $H^1(\Gamma;\real) = 0$ (and $\Gamma$ is finitely generated), we know $\alpha$ is trivial on some finite-index subgroup~$\Gamma'$ of~$\Gamma$.

Then the bounded Euler class $\delta_1 \alpha$ is trivial on~$\Gamma'$, so \cref{BddEulerClassFP} tells us that $\Gamma'$ has a fixed point~$p$.
Since $\Gamma'$ has finite index, we see that the $\Gamma$-orbit of~$p$ is finite.
\end{proof}


\begin{thm}[Ghys \cite{GhysCercle}, Burger-Monod \cite{BurgerMonod-BddCohoLatts}]
If\/ $\Gamma$ is any arithmetic subgroup of\/ $\SL(n,\real)$, with $n \ge 3$, then every action of\/~$\Gamma$ on~$S^1$ has a finite orbit.
\end{thm}

\begin{proof}[Outline of Burger-Monod proof]
Burger and Monod showed (in a much more general setting) that the comparison map $\Hb^2(\Gamma;\real) \to H^2(\Gamma;\real)$ is injective. 
Since it is known that $H^2(\Gamma;\real) = 0$ (if $n$~is sufficiently large), we conclude that $\Hb^2(\Gamma;\real) = 0$.

The other hypotheses of \cref{H2RVanishFP} are well known to be true.
\end{proof}

\section{Computing $\Hb^2(\Gamma;\real)$}

To calculate $\Hb^2(\Gamma;\real)$, we would like to understand the kernel of the comparison map $\Hb^2(\Gamma;\real) \to H^2(\Gamma;\real)$. For this, we introduce some notation:

\begin{defn} \ 
\noprelistbreak
\begin{itemize}
\item A function $\alpha \colon \Gamma \to \real$ is a:
	\begin{itemize} 
	\item \emph{quasimorphism} if 
	$\alpha(gh) - \alpha(g) - \alpha(h)$ is bounded (as a function of $(g,h) \in \Gamma \times \Gamma$);
	\item \emph{near homomorphism} if it is within a bounded distance of a homomorphism.
	\end{itemize}
\item We use $\QM(\Gamma,\real)$ and $\NH(\Gamma,\real)$ to denote the space of quasimorphisms and the space of near homomorphisms, respectively.
\end{itemize}
Note that $\NH(\Gamma,\real) \subset \QM(\Gamma,\real)$.
\end{defn}

\begin{prop} \label{KernelComparison}
The kernel of the comparison map $\Hb^2(\Gamma;\real) \to H^2(\Gamma;\real)$ is 
	$$ \frac{\QM(\Gamma,\real)}{\NH(\Gamma,\real)} .$$
\end{prop}

\begin{proof}
Let $c$ be a bounded $2$-cocycle, such that $c$ is trivial in $H^2(\Gamma;\real)$. 
Then $c = \delta_1 \alpha$, for some $\alpha \colon \Gamma \to \real$. 
Thus, for all $g,h \in \Gamma$, we have
	$$ \text{$| \alpha(gh) - \alpha(g) - \alpha(h)| = | \delta_1 \alpha \, (g,h)| = |c(g,h)| \le \| c\|_\infty$ is bounded} .$$
So $\alpha$ is a quasimorphism. This establishes that $\QM(\Gamma,\real)$ maps onto the kernel of the comparison map, via $\alpha \mapsto \delta_1 \alpha$.

Now suppose $\alpha \in \QM(\Gamma,\real)$, such that $\delta_1 \alpha$ is trivial in $\Hb^2(\Gamma;\real)$. The triviality of $\delta_1 \alpha$ means there is a bounded function $c \colon \Gamma \to \real$, such that $\delta_1 \alpha = \delta_1 c$.
Then $\delta_1(\alpha - c) = 0$, so $\alpha - c$ is a homomorphism.
Since $c$ is bounded, this means that $\alpha$ is within a bounded distance of a homomorphism; i.e., $\alpha \in \NH(\Gamma,\real)$.
\end{proof}

\begin{eg}[Brooks \cite{Brooks-RemBddCoho}]
We can construct many quasimorphisms on the free group~$F_2$:
\begin{itemize}
\item As a warm-up, recall that there is an obvious homomorphism $\varphi_a$, defined by letting $\varphi_a(x)$ be the (signed) number of occurrences of~$a$ in the reduced representation of~$x$. For example,
	$$\varphi_a(a^2 b a^{3} b^2 a b^{-3} a^{-7} b^2) = 2 + 3 + 1 - 7 = -1 .$$
There is an analogous homomorphism~$\varphi_b$, and every homomorphism $F_2 \to \real$ is a linear combination of these two.

\item Similarly, for any nontrivial reduced word~$w$, we can let $\varphi_{w}(x)$ be the (signed) number of disjoint occurrences of~$w$ in the reduced representation of~$x$. For example, 
	$$\varphi_{ab}(a^2 b a^{3} b^2 a b^{-3} a^{-7} b^2) = 1 + 1 -1 = 1 .$$
This is a quasimorphism.
\end{itemize}
\end{eg}

\begin{exer} \label{WordQuasi}
Verify that $\varphi_{w}$ is a quasimorphism, for any reduced word~$w$.
\end{exer}

With these quasimorphisms in hand, it is now easy to prove a fact that was mentioned in \cref{BddCoho(F2)}:

\begin{exer} \label{H2b(free)}
Show that $\Hb^2(F_2;\real)$ is infinite-dimensional.
\hint{Verify that $\varphi_{a^k}$ ($k \ge 2$) is not within a bounded distance of the linear span of $\{\varphi_b, \varphi_a, \varphi_{a^{k+1}}, \varphi_{a^{k+2}}, \varphi_{a^{k+3}}, \ldots\}$, by finding a word~$x$, such that $\varphi_{a^k}(x)$ is large, but the others vanish on~$x$.}
\end{exer}

\begin{exers} \label{QuasiMExers}\ 
\noprelistbreak
\begin{enumerate}

\item \label{QuasiMExers-KernelFD}
Show that if $\Gamma$ is boundedly generated (by cyclic subgroups), then the kernel of the comparison map $\Hb^2(\Gamma;\real) \to H^2(\Gamma;\real)$ is finite-dimensional.
\hint{Every quasimorphism $\integer \to \real$ is a near homomorphism.}

\item \label{QuasiMExers-F2NotBddGen}
Show the free group~$F_2$ is not boundedly generated (by cyclic subgroups).
 
 \item \label{QuasiMExers-commutator}
 Show that every quasimorphism is bounded on the set of commutators $\{x^{-1} y^{-1} x y\}$.

 \item \label{QuasiMExers-amenable}
 Show that if $\Gamma$ is amenable group, and the abelianization of~$\Gamma$ is finite, then $\Gamma$ does not have unbounded quasimorphisms.
 \\ \hint{Amenable groups do not have bounded cohomology.}
 
 \item \label{QuasiMExers-SL3ZNoQuasi}
 Show that $\SL(3,\integer)$ has no unbounded quasimorphisms.
 \\ \hint{Use \cref{SL3ZBddGen}.}

\end{enumerate}
\end{exers}


\begin{appendix}
\chapter*{Hints for the exercises}

\soln{FreeGrpActsOnR}
Every finitely generated free group is a subgroup of the free group $F_2 = \langle a, b \rangle$, so we need only consider this one example.

Choose any faithful action of~$F_2$ on the circle. (For example, use the \emph{Ping-Pong Lemma} \cite[Lem.~2.3.9]{Navas-GrpsDiffeos}, 
which provides a sufficient condition for two homeomorphisms to generate a free group, or note that $F_2$ is a subgroup of $\PSL(2,\real)$, which acts faithfully on the circle by linear-fractional transformations.) Since $F_2$ is free, we can lift this to an action on the line (simply by choosing any lift of the two generators $a$ and~$b$). Since it projects down to a faithful action on the circle, this action on the line must also be faithful.

\soln{DirProdFaithful} Since any open interval is homeomorphic to~$\real$, we may let $\Gamma_1$ act on the open interval $(0,1)$ (fixing all points in the complement), and let $\Gamma_2$ act on the open interval $(1.2)$ (fixing all points in the complement). These actions commute, so they define an action of $\Gamma_1 \times \Gamma_2$.

\soln{ActIffLO} Details are in \cite[Thm.~6.8]{GhysCircleSurvey}.

\soln{ProdPos} By left-invariance, we have
	$$\text{
	$ab = a \cdot b \succ a \cdot e = a \succ e $
	\quad and \quad
	$e = a^{-1} \cdot a \succ a^{-1} \cdot e = a^{-1}$
	} .$$

\fullsoln{HeisExers}{z=[x,y]} Straightforward matrix multiplication verifies that $z = [x,y]$ and that $z$ commutes with both~$x$ and~$y$.

\fullsoln{HeisExers}{[xyyk]} Since $z = [x,y]$, we have $xy = yxz$. By induction on~$k$ (and using the fact that $z$ commutes with~$x$), then $x^k y = yx^k z^k$. By induction on~$\ell$ (and using the fact that $z$ commutes with~$y$), then $x^k y^\ell = y^\ell x^k (z^k)^\ell$ for $k,\ell \in \integer^+$. 

\fullsoln{HeisExers}{orderable} 
To apply \fullcref{LOExers}{extension}, note that $H$ has a chain of normal subgroups
	$$ \{e\} \normal \langle z \rangle \normal \langle z,x \rangle \normal H ,$$
and each quotient is isomorphic to~$\integer$ (hence, has an obvious left-invariant order). 


\fullsoln{SL3ZPfExers}{Heis} Either calculate that $\left[\ovalbox{$k-1$}, \ovalbox{$k+1$}\right] = \ovalbox{$k$}$, and that $\ovalbox{$k$}$ commutes with both $\ovalbox{$k-1$}$ and $\ovalbox{$k+1$}$, or observe that some permutation matrix conjugates the ordered triple $\left( \ovalbox{$k-1$} \, , \ovalbox{$k$} \, , \ovalbox{$k+1$} \right)$ to $(x,z,y)$.

\fullsoln{SL3ZPfExers}{FinInd} Details are in \cite[\S3]{Witte-QrankAct1mfld}.

\fullsoln{LOExers}{subgrp} The restriction of a left-invariant total order is a left-invariant total order.

\fullsoln{LOExers}{abelian} We may assume $\Gamma$ is finitely generated \fullcsee{LOExers}{locally}, so $\Gamma \iso \integer \times \cdots \times \integer$. Now use \cref{DirProdFaithful} or \fullcref{LOExers}{extension}.

\fullsoln{LOExers}{extension} Let $\prec_*$ and $\prec^*$ be left-invariant total orders on~$N$ and $\Gamma/N$, respectively. Then we define
	$$ g \prec h \quad \iff \quad  \begin{matrix}
	\text{$gN \prec^* hN$ \quad or} \\[5pt]
	\text{$gN = hN$ \ and \ $h^{-1} g \prec_* e$}
	. \end{matrix} $$
The left-invariance of $\prec_*$ and $\prec^*$ implies the left-invariance of~$\prec$.

\fullsoln{LOExers}{nilpotent} Recall that the \emph{ascending central series} 
	$$ \{e\} = Z_0 \normal Z_1 \normal \cdots \normal Z_c = \Gamma $$
is defined inductively by $Z_i/Z_{i-1} = Z \bigl( \Gamma/ Z_{i-1})$. 
Fix $i \ge 2$ and let $\overline\Gamma = \Gamma/Z_{i-2}$. For any nontrivial $z \in \overline{Z_i}$, there exists $g \in \Gamma$, such that $[\overline{g}, z]$ is a nontrivial element of~$\overline{Z_{i-1}}$, which is torsion-free (by induction). Since $[\overline{g}, z^n] = [\overline{g}, z]^n$, this implies that $Z_i/Z_{i-1}$ is torsion-free. It is also abelian, so we can apply \fullcref{LOExers}{abelian} and \fullcref{LOExers}{extension}.

\fullsoln{LOExers}{solvable} This is not at all obvious, but specific examples are given on page~52 of \cite{KopytovMedvedev}. 

Here is the general philosophy. Assume $G$ is a nontrivial, finitely generated, left-orderable group. If $G$ is solvable (or, more generally, if $G$ is ``amenable''), then it can be shown that the abelianization of~$G$ is infinite \cite{Morris-AmenOnLine}. So any torsion-free solvable group with finite abelianization provides an example.

\fullsoln{LOExers}{Exps} ($\Rightarrow$) Choose $\epsilon_i$ so that $g_i^{\epsilon_i} \succ e$. Then every element of the semigroup is $\succ e$.

($\Leftarrow$) The condition implies it is possible to choose a semigroup~$P$ in~$\Gamma$, such that $e \notin P$ and, for every nonidentity element~$g$ of~$\Gamma$, either $g \in P$ or $g^{-1} \in P$. Define $x \prec y \iff x^{-1} y \in P$. Details can be found in \cite[Thm.~3.1.1, p.~45]{KopytovMedvedev}.

\fullsoln{LOExers}{locally} Use \fullcref{LOExers}{Exps}.

\fullsoln{LOExers}{residually} Let $g_1,\ldots,g_n$ be  nontrivial elements of~$\Gamma$. For each~$i$, there is a left-orderable group~$H_i$, and a homomorphism $\varphi_i \colon \Gamma \to H_i$, such that $\varphi_i(g_i) \neq e$. For the resulting homomorphism~$\varphi$ into $H_1 \times \cdots \times H_n$, we have $\varphi(g_i) \neq e$ for all~$i$. Now apply \fullcref{LOExers}{Exps}.

\fullsoln{LOExers}{BurnsHale} Given nontrivial elements $g_1,\ldots,g_n$ of~$\Gamma$, the assumption provides a nontrivial homomorphism $\rho \colon \langle g_1,\ldots,g_n \rangle \to \real$. Assume $\rho$ is trivial on $g_1,\ldots,g_k$, and nontrivial on the rest. By induction on~$n$, we can choose $\epsilon_1,\ldots,\epsilon_k \in \{\pm1\}$, such that the semigroup generated by $\{g_1,\ldots,g_k\}$ does not contain~$e$. For $i > k$, choose $\epsilon_i$ so that $\rho(g_i^{\epsilon_i}) > 0$. Then the semigroup generated by $g_1,\ldots,g_n$ does not contain~$e$, so \fullcref{LOExers}{Exps} applies. (Details can be found on page~50 of \cite{KopytovMedvedev}.)

\fullsoln{BddGenExers}{quotient} If $\Gamma = H_1 \cdots H_n$, then $\Gamma/N = \overline{H_1} \cdots \overline{H_n}$, where $\overline{H_i}$ is the image of~$H_i$ in $\Gamma/N$.

\fullsoln{BddGenExers}{modnth} Call the subgroup~$N$. Then $N$~is a normal subgroup, so \fullcref{BddGenExers}{quotient} tells us that $\Gamma/N$ is boundedly generated by cyclic groups. However, every element of $\Gamma/N$ has finite order, so all of these cyclic groups are finite.

\fullsoln{BddGenExers}{FinInd} 
($\Leftarrow$) For a normal subgroup~$N$ of~$\Gamma$, it is easy to see that if $N$ and $\Gamma/N$ are boundedly generated, then $\Gamma$ is boundedly generated. Also note that finite groups are (obviously) boundedly generated.

($\Rightarrow$) To present the main idea with a minimum of notation, let us assume $\Gamma = H K$ is the product of just two cyclic groups. Let $\dot K = K \cap \dot\Gamma$, and let $\{k_1,\ldots,k_n\}$ be a set of coset representatives for $\dot K$ in~$K$. There exists a finite-index subgroup~$\dot H$ of~$H$, such that the conjugate $\dot H^{k_j}$ is contained in~$\dot\Gamma$ for every~$j$. Let $\{h_1,\ldots,h_m\}$ be a set of coset representatives for $\dot H$ in~$H$. Given $g \in \dot\Gamma$, we may choose $h \in h$ and $k \in K$, such that
	$$g = h k
	= (h_i \dot h) (k_j \dot k)
	= (h_i k_j) \dot h^{k_j} \dot k 
	.$$
Therefore, if we let $\ell_1,\ldots,\ell_r$ be a list of the elements in $\dot\Gamma \cap \{h_i k_j\}$, then
	$$ \dot\Gamma 
	= \langle \ell_1 \rangle \cdots \langle \ell_r \rangle 
	\ 
	\dot H^{k_1} \cdots \dot H^{k_n} 
	\ 
	\dot K 
	.$$

\fullsoln{BddGenExers}{SL2Z} Let $\Gamma$ be a free subgroup of finite index in $\SL(2,\integer)$. Then \fullcref{QuasiMExers}{F2NotBddGen} tells us that $\Gamma$ is not boundedly generated, so \fullcref{BddGenExers}{FinInd} implies that $\SL(2,\integer)$ also is not boundedly generated by cyclic groups. Since $\U$ and~$\V$ are cyclic, this completes the proof.

\fullsoln{BddGenExers}{VariablePowers} 
The argument is somewhat similar to \fullcref{BddGenExers}{FinInd}($\Rightarrow$). Let $\dot\Gamma$ be the subgroup that is under consideration, and let us assume, for simplicity, that $\Gamma = HK$ is the product of just two cyclic groups. Let $\dot K = K \cap \dot\Gamma$, and let $\{k_1,\ldots,k_n\}$ be a set of coset representatives for $\dot K$ in~$K$. 

The key point is to observe that, by definition, $\dot\Gamma$ contains a finite-index subgroup of each~$H^{k_j}$, so we may choose a finite-index subgroup~$\dot H$ of~$H$, such that $\dot H^{k_j}$ is contained in~$\dot\Gamma$ for every~$j$. Let $\{h_1,\ldots,h_m\}$ be a set of coset representatives for $\dot H$ in~$H$. For $g \in \Gamma$, we have
	$$g = h k
	= (h_i \dot h) (k_j \dot k)
	= (h_i k_j) \dot h^{k_j} \dot k 
	\in (h_i k_j) \dot\Gamma
	.$$
Therefore, $\{ h_i k_j \}$ contains a set of coset representatives, so the index of~$\dot\Gamma$ is at most~$mn$.

\soln{BddGenIsom} The Triangle Inequality implies that \emph{every} orbit of every cyclic group is bounded. Now, for any $x \in X$, any $R \in \real^+$, and any cyclic subgroup~$H_i$ of~$\Gamma$, this implies there exists $r_i \in \real^+$, such that $H_ix$ is contained in the ball $B_{r_i}(x)$. By the Triangle Inequality, we have $H_i \cdot B_R(x) \subseteq B_{R+r_i}(x)$. By induction, if $\Gamma = H_1 \cdots H_n$, then $\Gamma x \subseteq B_{r_1 + \cdots + r_n}(x)$.

\soln{AbelNoPonzi} Suppose $M$ is a Ponzi scheme on~$\Gamma$, and $M(g) \in gS$ for all~$g$. Let $k$ be the maximum word length of an element of~$S$. Then, since $M$ is (at least) 2-to-1, we know that the ball of radius $r + k$ has at least twice as many elements as the ball of radius~$r$. So $\Gamma$ has exponential growth.

\fullsoln{FolnerEx}{Free} Assume, without loss of generality, that (at least) $3/4$ of the elements of~$F$  do \emph{not} start with~$a^{-1}$. Then $3/4$ of~$aF$ starts with~$a$ and $3/4$ of~$baF$ starts with  $b$. If $F$ is almost invariant, this implies that almost half of~$F$ starts with both~$a$ and~$b$.

\fullsoln{FolnerEx}{union} Let $n = \#S$, and choose $F$ so that $\#(F \cap aF) > \bigl( 1 - (\epsilon/n) \bigr) \#F$ for all $a \in S$. Then $\#(aF \smallsetminus F) < \epsilon/n$ for all $a \in S$, so $\#(SF) - \#F < n \cdot (\epsilon/n) = \epsilon$.

\fullsoln{FolnerEx}{QI} It will suffice to prove the hint, since it gives a condition that is invariant under quasi-isometry.

Suppose $\Gamma$ is not amenable. Then there exist $S$ and~$\epsilon$, such that $\#(SF) \ge (1 + \epsilon) \#F$, for every finite subset~$F$ of~$\Gamma$. Choosing $n$~large enough that $(1 + \epsilon)^n > c$, then we have $\#(S^nF) \ge c \cdot \#F$.

The other direction is immediate from \fullcref{FolnerEx}{union}.

\fullsoln{FolnerEx}{noPonzi} Suppose $M$ is a Ponzi scheme on~$\Gamma$, and $M(g) \in gS$ for all~$g$. From \fullcref{FolnerEx}{union} (and replacing $F$ with~$F^{-1}$ to convert left-translations into right-translations), we know there is a finite set~$F$, such that $\#(F S^{-1}) < 2 \cdot \#F$. This is impossible, since $M$ is (at least) 2-to-1 and $M^{-1}(F) \subseteq F S^{-1}$.

\fullsoln{FolnerEx}{Ponzi} In the special case where $a_i = 1$ for all~$i$, the hint is known as \emph{Hall's Marriage Theorem}, and can be found in many combinatorics textbooks. The general case is proved similarly.
Since there are only finitely many possibilities for each set~$A_i'$, a standard diagonalization argument shows that the result is also valid for an infinite sequence $A_1,A_2,\ldots$ of finite sets and an infinite sequence $\{a_i\} \subseteq \natural$.

From the hint to \fullcref{FolnerEx}{QI}, there exists $S \subset \Gamma$, such that $\#(FS^{-1}) \ge 2 \cdot \#F$ for every finite $F \subset \Gamma$. For $y \in \Gamma$, let $A_y = y S^{-1}$ and $a_y = 2$. Then there exists $A_y' \subseteq A_y$, such that $\#A_y' = 2$ and the sets $\{A_y'\}_{y \in \Gamma}$ are pairwise disjoint. Define $M(g) = y \in g S$ for all $g \in A_y'$.

\fullsoln{VonNeumannEx}{noParadox}
Let
	$$ a = \sum_{i = 1}^m \mu(A_i) \text{\qquad and\qquad} b = \sum_{j = 1}^n \mu(B_j) ,$$
where $\mu$ is a finitely additive, translation-invariant probability measure on~$\Gamma$. Then, since $A_1,\ldots,A_m,B_1,\ldots,B_n$ are pairwise disjoint, we have
	$$a + b = \mu(\Gamma) = 1 .$$
On the other hand, since $\Gamma = \bigcup_{i=1}^m g_i A_i $, we have
	$$ 1 = \mu(\Gamma) 
	= \mu \left( \bigcup_{i=1}^m g_i A_i \right)
	\le \sum_{i = 1}^m \mu( g_i A_i )
	= \sum_{i = 1}^m \mu( A_i )
	= a ,$$
and, similarly, $b = 1$. This is a contradiction.

\fullsoln{VonNeumannEx}{free}
Let $A_1$, $A_2$, $B_1$, and~$B_2$ be the reduced words that start with $a$, $a^{-1}$, $b$, or~$b^{-1}$, respectively. (Also add $e$ to one of these sets.) Then 
	$$ F_2 = a^{-1} A_1 \cup a A_2 = b^{-1} B_1 \cup b B_2 .$$

\fullsoln{VonNeumannEx}{Paradox}
Let $M$ be a Ponzi scheme, and choose $S$ such that $M(g) \in Sg$. Let $A$ contain a single element of $M^{-1}(x)$, for every $x \in \Gamma$, and let $B$ be the complement of~$A$. Then, for $s \in S$, let
	$ A_s = \{\, g \in A \mid M(g) = sg \,\}$ 
	and
	$ B_s = \{\, g \in B \mid M(g) = sg \,\}$.
By construction, these sets are pairwise disjoint, and we have $\Gamma = \bigcup_{s \in S} s A_s = \bigcup_{s \in S} s B_s$.

\fullsoln{EgAmenEx}{finite} This is obvious from almost any characterization of amenability. For example, letting $F = \Gamma$ yields a nonempty, finite set that is invariant, not merely almost-invariant.

\fullsoln{EgAmenEx}{cyclic} By \fullcref{EgAmenEx}{finite}, it suffices to consider the infinite cyclic group~$\integer$. A long interval $\{0,1,2,\ldots,n\}$ is almost-invariant.

\fullsoln{EgAmenEx}{product} The Cartesian product of two almost-invariant sets is almost-invariant.

\fullsoln{EgAmenEx}{abelian} We may assume $\Gamma$ is finitely generated (by \fullcref{EgAmenEx}{locally}), so it is a direct product of finitely many cyclic groups.

\fullsoln{EgAmenEx}{extension} This is difficult to do with almost-invariant sets, because multiplication by an element of $G/N$ will act by conjugation on~$N$, which may cause distortion.

It is perhaps easiest to apply \cref{AmenOnConvex}. Since $N$ is amenable, it has fixed points in~$C$. The set $C^N$ of such fixed points is a closed, convex subset.  Also, it is $\Gamma$-invariant (because $N$ is normal). So $\Gamma$ acts on~$C^N$. Since $N$ is trivial on this set, the action factors through to $\Gamma/N$, which must have a fixed point. This is a fixed point for~$\Gamma$.

\fullsoln{EgAmenEx}{solv} By definition, a solvable group is obtained by repeated extensions of abelian groups, so this follows from repeated application of \fullcref{EgAmenEx}{extension}.

\fullsoln{EgAmenEx}{subgrp} This is another case that is difficult to do with almost-invariant sets. Instead, note that if there is a Ponzi scheme on some subgroup of~$\Gamma$, then it could be reproduced on all of the cosets, to obtain a Ponzi scheme on all of~$\Gamma$. This establishes the contrapositive.

\fullsoln{EgAmenEx}{quotient} This is immediate from \cref{ActHasInvtMeas} (or \cref{AmenOnConvex}), because any action of $\Gamma/N$ is also an action of~$\Gamma$. It also follows easily from \cref{VonNeumann}, since any subset of $\Gamma/N$ pulls back to a subset of~$\Gamma$.

\fullsoln{EgAmenEx}{locally} Given $S$ and~$\epsilon$, let $H$ be the subgroup generated by~$S$, so $H$ is finitely generated. If $H$ is amenable, then it contains an almost-invariant set, which is also an $(S,\epsilon)$-invariant set in~$G$.

\fullsoln{EgAmenEx}{limit} This is immediate from \fullcref{EgAmenEx}{locally}, because any finite subset of $\bigcup \mathcal{A}$ must be contained in one of the sets in~$\mathcal{A}$.

\fullsoln{EgAmenEx}{subexp} See the hint to \fullcref{FolnerEx}{QI}.

\soln{EulerTrivial} ($\Leftarrow$)~By assumption, we may choose the lifts in such a way that $\lift g \lift h = \lift{gh}$ for all $g$ and~$h$. So $c = 0$.

($\Rightarrow$) We have $c(g,h) = \varphi(gh) - \varphi(g) - \varphi(h)$ for some $\varphi \colon \Gamma \to \integer$. Then, letting $\widetilde\rho(g) (t) = \lift g(t) + \varphi(g)$, we have $\widetilde\rho( g ) \, \widetilde\rho( h )  = \widetilde\rho(gh)$, so $\widetilde\rho$ is a homomorphism.

\soln{EulerIsBdd} 
We have $c(g,h) = \lift g \bigl( \lift h(0) \bigr) - \lift{gh}(0)$. Note that
	\begin{itemize}
	\item $0 \le \lift{gh}(0) < 1$,
	and
	\item $ 0 \le \lift g(0) \le \lift g \bigl( \lift h(0) \bigr) < \lift g(1) = \lift g(0) + 1 < 1 + 1 = 2$,
	\end{itemize}
so both terms on the right-hand side are bounded.

\soln{WordQuasi}
In fact $\varphi_w(xy)$ never differs by more than~$1$ from $\varphi_w(x) + \varphi_w(y)$. There is a difference only if some occurrence of~$w$ (or~$w^{-1}$) overlaps the boundary between $x$ and~$y$, and there cannot be two such occurrences that are disjoint.

\soln{H2b(free)}
Let $x = (a^kbab^{-1})^n(a^{-(k-1)}b^2a^{-1}b^{-1}a^{-1}b^{-1})^n$.

\fullsoln{QuasiMExers}{KernelFD}
Write $\Gamma = H_1 \cdots H_r$. Any quasimorphism on~$\Gamma$ is determined, up to bounded error, by its restriction to the cyclic subgroups $H_1,\ldots,H_r$. 
Also, it is not difficult to show that every quasimorphism $\integer \to \real$ is a near homomorphism. (Or this can be deduced from \cref{Hbb(amen;R)=0} and \cref{KernelComparison}.) So the restriction of~$\varphi$ to each~$H_i$ is a near homomorphism. Since the homomorphisms from~$\integer$ to~$\real$ form a one-dimensional space, we conclude that the dimension of $\QM(\Gamma;\real)/\ell^\infty(\Gamma;\real)$ is at most~$r$.

\fullsoln{QuasiMExers}{F2NotBddGen} Compare \cref{H2b(free)} with \fullcref{QuasiMExers}{KernelFD}.

\fullsoln{QuasiMExers}{commutator} We have
	\begin{align*} \varphi ( x^{-1} y^{-1} x y )
	&= \varphi ( x^{-1} y^{-1} ) + \varphi( x y \bigr) \pm C
	= \varphi ( x^{-1} ) + \varphi(  y^{-1} ) + \varphi( x ) + \varphi(  y \bigr) \pm 3C
	\\&= \varphi ( x^{-1} x ) + \varphi(  y^{-1} y ) \pm 5C
	=  2\varphi ( e ) \pm 5C
	, \end{align*}
so $| \varphi ( x^{-1} y^{-1} x y )| \le 2 |\varphi ( e )| + 5C$.

\fullsoln{QuasiMExers}{amenable} Let $\varphi \colon \Gamma \to \real$ be a quasimorphism. From \cref{Hbb(amen;R)=0} and \cref{KernelComparison}, we see that $\varphi$ is within bounded distance of a homomorphism. However, since the abelianization of~$\Gamma$ is finite, there are no nontrivial homomorphisms $\Gamma \to \real$. Therefore $\varphi$ is bounded.

\fullsoln{QuasiMExers}{SL3ZNoQuasi} Every elementary matrix is a commutator (recall that $[x^k, y] = z^k$), so \fullcref{QuasiMExers}{commutator} implies that $\varphi$ is bounded on the set of elementary matrices. Since every element of $\SL(3,\integer)$ is the product of a bounded number of these elementary matrices (see \cref{SL3ZBddGen}), a simple estimate shows that $\varphi$ is bounded.

\end{appendix}


\end{document}